\documentclass[12pt,twoside,reqno]{amsart}
\usepackage[all]{xy} 
\usepackage{enumitem}  
\usepackage {amssymb,latexsym,amsthm,amsmath}
\usepackage[colorlinks, linkcolor=blue, citecolor=red]{hyperref}
\usepackage{mathtools}
\usepackage[utf8]{inputenc}
\usepackage{tikz}
\usepackage[english]{babel}

\long\def\symbolfootnote[#1]#2{\begingroup%
\def\thefootnote{\fnsymbol{footnote}}\footnote[#1]{#2}\endgroup}

\makeatletter
\def\imod#1{\allowbreak\mkern10mu({\operator@font mod}\,\,#1)}
\makeatother

\newtheorem{theorem}{Theorem}[section]
\newtheorem{lemma}[theorem]{Lemma}
\newtheorem{corollary}[theorem]{Corollary}

\theoremstyle{definition}
\newtheorem{remark}[theorem]{Remark}

\newtheorem{example}[theorem]{Example}

\numberwithin{equation}{section}

\begin{document}

\setcounter{section}{0}
\setcounter{tocdepth}{1}
% document information
\title[Twisted Conjugacy in Linear Algebraic Groups II]{Twisted Conjugacy in Linear Algebraic Groups II}
\dedicatory{A tribute to James Edward Humphreys (1939 - 2020).}
\author[Sushil Bhunia]{Sushil Bhunia}
\author[Anirban Bose]{Anirban Bose}
\thanks{Bose is supported by DST-INSPIRE Faculty fellowship(IFA DST/INSPIRE/04/2016/001846)}
\address{Indian Institute of Science Education and Research (IISER) Mohali, Knowledge City,  Sector 81, S.A.S. Nagar 140306, Punjab, India}
\email{sushilbhunia@gmail.com}
\address{SRM University, A.P. Neerukonda, Mangalagiri Mandal
Guntur District, Mangalagiri, Andhra Pradesh 522240}
\email{anirban.math@gmail.com}
\date{\today}

\begin{abstract}
Let $G$ be a linear algebraic group over an algebraically closed field $k$ and $\mathrm{Aut}_{\mathrm{alg}}(G)$ the group of all algebraic group automorphisms of $G$. For every $\varphi\in \mathrm{Aut}_{\mathrm{alg}}(G)$ let $\mathcal{R}(\varphi)$ denote the set of all orbits of the $\varphi$-twisted conjugacy action of $G$ on itself (given by $(g,x)\mapsto gx\varphi(g^{-1})$, for all $g,x\in G$). We say that $G$ has the algebraic $R_\infty$-property if $\mathcal{R}(\varphi)$ is infinite for every $\varphi\in \mathrm{Aut}_{\mathrm{alg}}(G)$. In \cite{bb} we have shown that this property is satisfied by every connected non-solvable algebraic group. From a theorem due to Steinberg it follows that if a connected algebraic group $G$ has the algebraic $R_\infty$-property, then $G^\varphi$ (the fixed-point subgroup of $G$ under $\varphi$) is infinite for all $\varphi\in \mathrm{Aut}_{\mathrm{alg}}(G)$. In this article we show that the condition is also sufficient. We also show that a Borel subgroup of any semisimple algebraic group has the algebraic $R_\infty$-property and identify certain classes of solvable algebraic groups for which the property fails.
\end{abstract}
\maketitle

\section*{Introduction}
Let $G$ be a group and $\varphi$ an automorphism of $G$. The \emph{$\varphi$-twisted conjugacy} action of $G$ on itself is defined as the map $G\times G\rightarrow G$ given by $(g,x)\mapsto gx\varphi(g^{-1})$, for all $g,x\in G$. Let $\mathcal{R}(\varphi)$  be the set of all orbits of this action and $R(\varphi)$ the cardinality of $\mathcal{R}(\varphi)$. An orbit $[x]_\varphi$ ($x\in G$) under the twisted action is also called the \emph{Reidemeister class} of $x$. The reason for this nomenclature is probably because the study of such actions can be traced back to the Nielsen-Reidemeister fixed point theory (c.f. \cite{jiang}). In what follows, $R(\varphi)=\infty$ (respectively, $R(\varphi)<\infty$) will mean that the set $\mathcal{R}(\varphi)$ is infinite (respectively, finite).

A group $G$ is said to have the \emph{$R_\infty$-property} if $R(\varphi)=\infty$ for every automorphism  $\varphi$ of $G$. The study of groups with this property has its origin in \cite{fh94}. The reader may refer to \cite{ft} for an overview and more literature.  Some recent works in this direction include \cite{bdr}, \cite{ms2020}, \cite{timur19}, and  \cite{gsw21}, where the $R_\infty$-property has been studied for twisted Chevalley groups, for the general and special linear groups over certain subrings of  $\overline{\mathbb{F}_p}(t)$, for unitriangular groups over an integral domain, and for fundamental groups of geometric $3$-manifolds, respectively. In the realm of linear algebraic groups, an early instance of considering the notion of twisted conjugacy appears in \cite{gant}. The reader is urged to look at \cite{St}, \cite{mohr03}, \cite{mw04}, and \cite{springer} for a host of interesting results. 

A linear algebraic group $G$ over an algebraically closed field, is said to have the \emph{algebraic $R_\infty$-property} if $R(\varphi)=\infty$ for every algebraic group automorphism $\varphi$ of $G$. In the sequel an algebraic group will always mean a linear algebraic group over an algebraically closed field, an automorphism $\varphi$ of an algebraic group $G$, will mean an abstract automorphism such that $\varphi$ and $\varphi^{-1}$ are morphisms of the underlying affine variety of the group, and the group of all such automorphisms will be denoted by $\mathrm{Aut}_{\mathrm{alg}}(G)$. In a previous paper \cite[Corollary 18]{bb} it has been shown that if $G$ is an algebraic group such that its connected component $G^\circ$ is non-solvable, then $G$ has the algebraic $R_\infty$-property. The aim of the present paper is to study this property for solvable algebraic groups.

In Section \ref{sg} we show that if $G$ is a connected solvable algebraic group which admits an automorphism $\varphi$ such that $R(\varphi)<\infty$, then the $\varphi$-twisted action is necessarily transitive (Theorem \ref{solv}). From a theorem due to Steinberg \cite[Theorem 10.1]{St} it follows that if a connected algebraic group $G$ has the algebraic $R_\infty$-property, then the fixed-point subgroup $G^\varphi$ is infinite for every automorphism $\varphi$ of $G$. We deduce that the condition is also sufficient (Theorem \ref{iff}). We also prove that if $G$ is a Borel subgroup of a semisimple algebraic group, then it has the algebraic $R_\infty$-property (Theorem \ref{borel}).

A  unipotent algebraic group of Chevalley type is defined as the unipotent radical of a Borel subgroup of a simple algebraic group (equivalently, a maximal connected unipotent subgroup of a simple algebraic group). Let $G$ be such a group and assume that the characteristic of the base field is different from $2$ and $3$. From the works of Fauntleroy \cite{faunt} and Gibbs \cite{gibbs} one obtains a description of all  automorphisms of $G$. We derive a necessary and sufficient condition for an automorphism $\varphi$ of $G$, for which  $R(\varphi)=1$ (Theorem \ref{maxunipotent}).

In Section \ref{ex} we compute $R(\varphi)$ for certain automorphisms $\varphi$ of some solvable algebraic groups. We observe that tori, groups of the form $\mathbb{G}_a^n$ and the $n$-dimensional Witt groups fail to have the algebraic
$R_\infty$-property for all $n\geq 1$. A connected nilpotent algebraic group has the algebraic $R_\infty$-property if and only if its unipotent radical has this property. Example \eqref{ex3} describes two distinct semidirect products of  $\mathbb{G}_m^n$ ($n\geq 1$) and $\mathbb{G}_a^r$ ($r\geq 2$) such that one of them has the algebraic $R_\infty$-property, while the other does not.

It has been shown in \cite{timurjaa, nas20} that if $k$ is an algebraically closed field of infinite transcendence degree over $\mathbb{Q}$, and $G$ is one of the groups $\mathrm{GL}_n(k)$, $\mathrm{SO}_n(k)$ or $\mathrm{Sp}_n(k)$, then there exists an abstract automorphism $\varphi$ of $G$ (induced by a non-trivial automorphism of $k$) such that $R(\varphi)=1$. The proof of this result has been carried out on a case by case basis. Therefore it is desirable to have an argument which may possibly work for any reductive algebraic group. This consideration forms a part of our ongoing work. However, it turns out that the proof of \cite[Theorem 6]{timurjaa} can be modified to show that if $k$ is an algebraically closed field of countable transcendence degree over $\mathbb{Q}$, then a Borel subgroup of any simple algebraic group over $k$, admits an abstract automorphism $\varphi$ such that $R(\varphi)=1$ (Theorem \ref{abstbor}).

\section{Preliminaries}\label{prel}
In this section we fix some notations and terminologies which will be used throughout the paper. Fix an algebraically closed field $k$. By an \emph{algebraic group} (over $k$) we mean a Zariski-closed subgroup of $\mathrm{GL}_n(k)$, for some $n\geq 1$. If $G$ is such a group, then its irreducible (equivalently, connected) component $G^\circ$ containing the identity is a closed normal subgroup of finite index in $G$; we say that $G$ is connected if $G=G^\circ$. An algebraic group $G$ is said to be \emph{solvable} if $\mathcal{D}^n(G)=e$ for some $n\geq 0$, where $\mathcal{D}^0(G):=G$ and $\mathcal{D}^{i+1}(G):=[\mathcal{D}^i(G),\mathcal{D}^i(G)]$, for all $i\geq 0$. For any connected solvable algebraic group $G$, there exist subgroups $T$ and $U$ such that $G=T\ltimes U$, where $T$ is a maximal torus and $U$ is the subgroup of all unipotent elements of $G$.

Now let $G$ be any connected algebraic group. The \emph{solvable radical} $R_s(G)$ is defined as the largest connected normal solvable subgroup of $G$ and the \emph{unipotent radical} $R_u(G)$ is defined as the largest connected normal unipotent subgroup of $G$. We say that $G$ is \emph{semisimple} (respectively, \emph{reductive}) if $R_s(G)=e$ (respectively, $R_u(G)=e$). A \emph{Borel subgroup} of $G$ is defined as a maximal closed connected solvable subgroup of $G$. We say that $G$ is \emph{simple} if it is not commutative and does not contain a non-trivial proper closed connected normal subgroup. It is known that every connected semisimple algebraic group (over $k$) is obtained as a Chevalley group based on $k$. Detailed constructions can be obtained from Steinberg's book \cite{St2}. The reader may also refer to \cite[Section 1.1]{bb} for a brief discussion leading to the definition of a Chevalley group. For basic properties of algebraic groups, one may refer to \cite{hum} or \cite{sp}.

Let $G$ be a connected semisimple algebraic group over $k$ and $\Phi$ the associated root system. Fix an arbitrary ordering on $\Phi$.  Viewing $G$  as a Chevalley group of type $\Phi$ based on $k$, one knows that it is generated by a subset $\{x_\alpha(t): \alpha\in \Phi,t\in k\}\subset G$. We record some properties of these generators :

\vspace*{3mm}

\noindent 1. For any $\alpha\in \Phi$ and $t,u\in k$, 
\begin{align}
x_\alpha(t)x_\alpha(u)=x_\alpha(t+u).
\end{align}

\noindent 2. \emph{Chevalley's commutator formula}: For any $\alpha,\beta\in \Phi$ and $t,u\in k$, 
\begin{align}\label{commutator}
x_{\alpha}(t)x_{\beta}(u)=x_{\beta}(u)x_{\alpha}(t)\prod_{\substack{i,j>0\\i\alpha+j\beta\in \Phi}}x_{i\alpha+j\beta}((-1)^{i+j}c_{ij}t^iu^j),
\end{align}
where the product on the right hand side is taken over all roots in the chosen ordering of $\Phi$ and $c_{ij}\in\{\pm1,\pm2,\pm3\}$ (depending on $\alpha,\beta$ and the ordering of $\Phi$).

\noindent 3. For any $\alpha\in \Phi$ and $t\in k^\times$, set
\begin{align}
 n_{\alpha}(t)=&x_{\alpha}(t)x_{-\alpha}(-t^{-1})x_{\alpha}(t), \\
 h_\alpha(t)=&n_{\alpha}(t)n_{\alpha}(-1).
\end{align}
Then $h_\alpha(t)h_\alpha(s)=h_\alpha(ts)$, for all $t,s\in k^\times$.
The subgroup $T=\langle h_\alpha(t):\alpha\in \Phi, t\in k^\times\rangle$ is a maximal torus of $G$. In fact, if $\Delta$ is a simple subsystem of $\Phi$, and $\Phi^+$ is the positive subsystem determined by $\Delta$, then $T=\langle h_\alpha(t): \alpha\in \Delta, t\in k^\times\rangle$. A maximal closed unipotent subgroup of $G$ is given by  $U:=\langle x_\alpha(t): t\in k,\alpha\in\Phi^+\rangle$ and $B=TU$ is a Borel subgroup of $G$. Also, every element of $U$ can be expressed uniquely as $\prod\limits_{\alpha\in \Phi^+}x_\alpha(t_\alpha)$ (for some $t_\alpha\in k$), the product being taken according to the fixed ordering on $\Phi^+$.

\noindent 4. For any $\alpha,\beta\in \Phi$, $t\in k^\times, u\in k$,
\begin{align}\label{1.5}
h_\alpha(t)x_\beta(u)h_\alpha(t)^{-1}=x_\beta(t^{\langle\beta,\alpha\rangle}u),
\end{align}
where $\langle\beta,\alpha\rangle:=\frac{2(\beta,\alpha)}{(\alpha,\alpha)}\in \mathbb{Z}$, and $(\cdot,\cdot) $ denotes the standard bilinear form on the Euclidean space spanned by $\Phi$.

Next, we collect some useful results.
\begin{lemma}\label{mindimclosed}
Let $G$ be an algebraic group acting morphically on an affine variety $X$. Then 
\begin{enumerate}
\item\cite[Proposition 8.3]{hum}\label{mindim} orbits of minimal dimension are closed;
\item\cite[Proposition 2.4.14]{sp}\label{uniclosed} if $G$ is unipotent, then all $G$-orbits in $X$ are closed.
\end{enumerate}
\end{lemma}

\begin{lemma}\label{autsimple}\cite[Lemma 7]{bb} If $G$ is a simple algebraic group, then $\mathrm{Aut}_{\mathrm{alg}}(G^n)\cong S_n\ltimes(\mathrm{Aut}_{\mathrm{alg}}(G))^n$ ($n\geq 1$), where $S_n$ denotes the group of all permutations on $n$ symbols.
\end{lemma}

\begin{lemma}\label{nonsolv}\cite[Corollary 18]{bb}
Let $G$ be an algebraic group such that $G^\circ$ is non-solvable. Then $G$ has the algebraic $R_\infty$-property.
\end{lemma}
\begin{lemma}\label{solvtorus}\cite[Proposition 20]{bb}
Let $G$ be a connected solvable algebraic group and $T$ a maximal torus of $G$. Suppose that $\varphi(T)=T$ implies $R(\varphi|_{T})=\infty$ for all $\varphi\in \mathrm{Aut}_{\mathrm{alg}}(G)$. Then $G$ has the algebraic $R_{\infty}$-property.
\end{lemma}

\begin{lemma}\label{stalg}\cite[Theorem 10.1]{St} 
      Let $G$ be a connected algebraic group, and $\varphi:G\rightarrow G$ a surjective homomorphism of algebraic groups. Then $|G^\varphi|<\infty$ implies that $R(\varphi)=1$.
\end{lemma}
\begin{lemma}\label{inner}\cite[Lemma 5]{bb}
	Let $\varphi$ be an automorphism of an algebraic group $G$ and $\mathrm{Int}_g$ the inner automorphism defined by   $g\in G$. Then $R(\varphi\circ\mathrm{Int}_g)=R(\varphi)$. In particular, $R(\mathrm{Int}_g)=R(\mathrm{Id})$, i.e., the number of inner twisted conjugacy classes in $G$ is equal to the number of conjugacy classes in $G$.
\end{lemma}

\begin{lemma}\label{ses}\cite[Lemma 6(1)]{bb}
	Let $e\longrightarrow N \overset{i}{\longrightarrow} G\overset{\pi}{\longrightarrow} Q\longrightarrow e$ be  an exact sequence of algebraic groups, and $\varphi\in\mathrm{Aut}_{\mathrm{alg}}(G)$ be such that $\varphi(N)=N$. Let $\overline{\varphi}$ denote the automorphism of $Q$ induced by $\varphi$. Then  $R(\varphi)\geq R(\overline{\varphi})$.
	
\end{lemma}

\begin{remark}
Lemma \ref{inner} and Lemma \ref{ses} are in fact true for any abstract automorphism $\varphi$ of a group $G$.
\end{remark}

\section{Results on algebraic groups}\label{sg}
We begin with the following basic result.
\begin{lemma}\label{unitorus}
Let $G$ be a connected algebraic group and $\varphi\in \mathrm{Aut}_{\mathrm{alg}}(G)$. If $G$ is unipotent or commutative, then $R(\varphi)\in\{1,\infty\}$. 
\end{lemma}
\begin{proof}
Since $G$ is connected (in particular, irreducible), it suffices to show that all the $\varphi$-conjugacy classes in $G$ are closed. The fact that it is true for a connected unipotent group $G$, follows from Lemma \ref{mindimclosed}\eqref{uniclosed}.

So let $G$ be commutative. We note that for any $x\in G$, its orbit is given by $[x]_\varphi=\{gx\varphi(g^{-1}): g\in G\}=x\{g\varphi(g^{-1}): g\in G\}=x[e]_\varphi$. Therefore, all the orbits have same dimension and  hence, each of them is closed by Lemma \ref{mindimclosed}(\ref{mindim}).
\end{proof}

Let $T$ be an $n$-dimensional torus ($n\geq 1$). By Lemma \ref{unitorus}, if $R(\varphi)<\infty$ for some $\varphi\in \mathrm{Aut}_{\mathrm{alg}}(T)$, then the $\varphi$-conjugacy action is transitive. We deduce a couple of conditions on $\varphi$, which are equivalent to the fact that $R(\varphi)=1$. Without loss of generality assume that $T=\mathbb{G}_m^n$ and identify  $\mathrm{GL}_n(\mathbb{Z})$ with $\mathrm{Aut}_{\mathrm{alg}}(T)$ via
$A=\begin{pmatrix}
a_{ij}
\end{pmatrix} \mapsto \varphi_A$; the automorphism $\varphi_A$ being defined by $\varphi_A((t_1,\ldots, t_n))=(s_1,\ldots, s_n)$, where $s_i=\prod\limits_{j=1}^nt_j^{a_{ij}}$ for all $t_i\in \mathbb{G}_m (1\leq i\leq n)$. A matrix in $\mathrm{SL}_{n}(\mathbb{Z})$ is called \emph{elementary} if it is of the form $E_{ij}(c)\; (1\leq i\neq j\leq n, c\in\mathbb{Z})$, where $(i,i)^{th}$ entry is $1$ for all $i$, $(i,j)^{th}$ entry is $c$ and every other entry is zero. With this notation, we have the following.

\begin{theorem}\label{torus}
	For an element $\varphi\in\mathrm{Aut}_{\mathrm{alg}}(T)$, the following are equivalent:
	\begin{enumerate}
		\item $R(\varphi)=1$.
		\item $\det(A-\mathrm{Id})\neq 0$, where $A\in \mathrm{GL}_n(\mathbb{Z})$ is such that $\varphi=\varphi_A$.
		\item The fixed-point subgroup $T^{\varphi}=\{t\in T: \varphi(t)=t\}$ is finite.
	\end{enumerate}
\end{theorem}
\begin{proof}
	\emph{(1) $\Leftrightarrow$ (2)}
	Let $\varphi=\varphi_A\in \mathrm{Aut}_{\mathrm{alg}}(T)$, for some $A=(a_{ij})\in \mathrm{GL}_n(\mathbb{Z})$ and assume that $R(\varphi)=1$. Thus $x\in [e]_\varphi=\{t^{-1}\varphi(t):t\in T\}$ for every $x\in T$. In other words, for every $x=(x_1,\ldots,x_n)\in T$ ($x_i\in \mathbb{G}_m$), there exists $t=(t_1,\ldots,t_n)\in T$ such that the following equations hold:
	\begin{align}\label{1}
	t_1^{a_{i1}}t_2^{a_{i2}}\cdots t_i^{a_{ii}-1}\cdots t_n^{a_{in}}=x_i,~ 1\leq i\leq n.
	\end{align}
	
	Treating $T$ as a $\mathbb{Z}$-module and writing Equation \eqref{1} additively we get
	
	\begin{align}
	a_{i1}t_1+\cdots+(a_{ii}-1)t_i+\cdots+a_{in}t_n=x_i,\; 1\leq i\leq n.
	\end{align}
	
	So we have a matrix equation 
	
	\begin{align}\label{2}\begin{pmatrix}a_{11}-1&\hdots&a_{1n}\\
	\vdots&\ddots&\vdots\\
	a_{n1}&\hdots&a_{nn}-1
	\end{pmatrix}
	\begin{pmatrix}
	t_1\\\vdots\\t_n
	\end{pmatrix}=\begin{pmatrix}
	x_1\\\vdots\\x_n
	\end{pmatrix}.
	\end{align}
	
	\noindent Now if possible let $\det(A-\mathrm{Id})=0$. Then by pre multiplying both sides of Equation \eqref{2}  by $\mathrm{Adj}(A-\mathrm{Id})$ we get
	
	\begin{align}\label{3}
	\mathrm{Adj}(A-\mathrm{Id})\begin{pmatrix}
	x_1\\\vdots\\x_n
	\end{pmatrix}=\begin{pmatrix}
	0\\\vdots\\0
	\end{pmatrix},
	\end{align} where the $0$ in Equation \eqref{3} denotes the zero element of the $\mathbb{Z}$-module $\mathbb{G}_m$. Then one can easily find suitable $x_i\in \mathbb{G}_m$ for which Equation \eqref{3} fails, a contradiction.  
	
	 Conversely, suppose that $\det(A-\mathrm{Id})\neq 0$. There exists elementary matrices $E_1,\ldots,E_l$ in $\mathrm{SL}_n(\mathbb{Z})$ such that $E_1\cdots E_l(A-\mathrm{Id})=(b_{ij})$, where $b_{ij}=0$ for all $i>j$ and $b_{ii}\neq 0 \;(1\leq i\leq n)$. Now since the base field is algebraically closed, for every $x=(x_1,\ldots,x_n)\in T$ there exists $t=(t_1,\ldots,t_n)\in T$ such that 
	\begin{align}
	(b_{ij})\begin{pmatrix}
	t_1\\\vdots\\t_n
	\end{pmatrix}=E_1\cdots E_l\begin{pmatrix}
	x_1\\\vdots\\x_n
	\end{pmatrix}.
	\end{align}
	Therefore $t$ and $x$ satisfy Equation \eqref{2} and hence Equation \eqref{1}, thereby showing that $R(\varphi)=1$.
	
	\vspace*{2mm}

	\emph{(3) $\Leftrightarrow$ (1)} If $T^{\varphi}$ is finite, then by Lemma \ref{stalg} we get $R(\varphi)=1$.
	On the other hand suppose that $R(\varphi)=1$. Since $\mathrm{dim}(T)-\mathrm{dim}(\mathrm{Stab}_T(e))=\mathrm{dim}([e]_\varphi)=\mathrm{dim}(T)$, we conclude that $T^\varphi=\mathrm{Stab}_T(e)$ is finite. 
\end{proof}

\begin{remark}

%The reader may see \cite{} and \cite{}, where statements analogous to Lemma \ref{} and Theorem \ref{} have been observed.

A result analogous to Lemma \ref{unitorus} (respectively, Theorem \ref{torus}) was observed in \cite[Proposition 3.2]{dekgon} (respectively, \cite[Lemma 4.1]{rom}) for a divisible abelian group (respectively, a free abelian group of finite rank).
\end{remark}

Next, we show that the conclusion of Lemma \ref{unitorus} is true if the group $G$ therein is assumed to be solvable.

\begin{theorem}\label{solv}
If $G$ is a connected solvable algebraic group and $\varphi\in \mathrm{Aut}_{\mathrm{alg}}(G)$, then $R(\varphi)\in\{1, \infty\}$.
\end{theorem}
\begin{proof}
Let $G=T\ltimes U$, where $T$ is a maximal torus and $U$ is the unipotent radical of $G$. Since $\varphi(T)$ is also a maximal torus, there exists $g\in G$ such that $g\varphi(T)g^{-1}=T$ and by Lemma \ref{inner}, $R(\varphi)=R(\mathrm{Int}_g\varphi)$. Thus without loss of any generality we assume that $\varphi(T)=T$. By virtue of Lemma \ref{unitorus}, $R(\varphi|_T)$ (respectively, $R(\varphi|_U)$) is either $1$ or $\infty$. So, we consider the following cases:

\vspace*{2mm}

\noindent\textbf{Case 1:} If $R(\varphi|_T)=R(\varphi|_U)=1$, then we \textbf{claim} that $R(\varphi)=1$. So take any element $sv\in G$, where $s\in T$, $v\in U$. Let $t\in T$ be such that $ts\varphi(t^{-1})=e$. Then $tsv\varphi(t^{-1})=ts\varphi(t^{-1})\varphi(t)v\varphi(t^{-1})=\varphi(t)v\varphi(t^{-1})=w $ (say). Note that $w\in U$ and therefore by our assumption, there exists $u\in U$ such that $uw\varphi(u^{-1})=e$. Thus $utsv\varphi(t^{-1}u^{-1})=e$ and this proves the claim.

\vspace*{2mm}

\noindent\textbf{Case 2:} Let $R(\varphi|_T)=1$ and $R(\varphi|_U)=\infty$. In this case we intend to show that $R(\varphi)=\infty$. So, if possible let there exist only finitely many $\varphi$-conjugacy classes in $G$. First, observe that just as in Case 1 above, every element of $G$ is $\varphi$-conjugate to an element of $U$. So let $[v_1]_\varphi,\ldots,[v_m]_\varphi$ be all the distinct classes in $G$ with the $v_i$'s in $U$. Since $R(\varphi|_T)=1$, by Theorem \ref{torus} let $T^\varphi=\{t_1,\ldots,t_n\}$ and set $S=\bigcup\limits_{i=1}^m S_i$, where $S_i:=\{[t_j^{-1}v_it_j]_{\varphi|_U}: 1\leq j\leq n\}$. Note that $S$ is a finite set of $\varphi|_U$-conjugacy classes in $U$. So let $v\in U$ be an arbitrary element. Then $v\sim_\varphi v_i$ for some $1\leq i\leq m$. So, there exists $t\in T, u\in U$ such that $v_i=tuv\varphi(u^{-1}t^{-1})=t\varphi(t^{-1})\varphi(t)uv\varphi(u^{-1})\varphi(t^{-1})$. Therefore, $t\varphi(t^{-1})=e$ which implies that $t\in T^\varphi$. So, if $t=t_l\in T^\varphi$, then $v_i=t_luv\varphi(u^{-1})t_l^{-1}$ which shows that $v\sim_{\varphi|_U} t_l^{-1}v_it_l$ or equivalently $[v]_{\varphi|_U}\in S$. Thus the finite set $S$ accounts for all the $\varphi|_U$-conjugacy classes of $U$ contrary to the assumption that $R(\varphi|_U)=\infty$.

\vspace*{2mm}

\noindent\textbf{Case 3:} If $R(\varphi|_T)=\infty$, then from the proof of Lemma \ref{solvtorus}, it follows that $R(\varphi)=\infty$. We provide an argument for the sake of completeness. For any $t\in T$, $[t]_\varphi=\{st\varphi(s^{-1})\varphi(s)t^{-1}ut\varphi(u^{-1})\varphi(s^{-1}):s\in T, u\in U\}$. Therefore, it is clear that if $R(\varphi)<\infty$, then $R(\varphi|_T)<\infty$ as required.

\vspace*{2mm}
This completes the proof.
\end{proof}
       Now let $G$ be a connected algebraic group. First, suppose that $G$ has the algebraic $R_\infty$-property. Then for every $\varphi\in \mathrm{Aut}_{\mathrm{alg}}(G)$, $R(\varphi)\neq 1$ and hence, by Lemma \ref{stalg}, $|G^\varphi|=\infty$. Conversely, suppose that there exists a $\varphi\in \mathrm{Aut}_{\mathrm{alg}}(G)$ for which $R(\varphi)<\infty$. Then by Lemma \ref{nonsolv}, $G$ is necessarily solvable. So, by Theorem \ref{solv}, we have $R(\varphi)=1$. Therefore  $\mathrm{dim}(G^\varphi)=\mathrm{dim}(G)-\mathrm{dim}([e]_\varphi)=0$ and hence, $G^\varphi$ is finite. We summarize this as
 
\begin{theorem}\label{iff}
       A connected algebraic group $G$ has the algebraic $R_{\infty}$-property if and only if the fixed point subgroup $G^{\varphi}$ is infinite for all $\varphi\in \mathrm{Aut}_{\mathrm{alg}}(G)$. 
\end{theorem}

Next, we record the following necessary and sufficient condition for a twisted conjugacy action to be transitive.

\begin{theorem}\label{general}
Let $G$ be a connected algebraic group and $\varphi\in \mathrm{Aut}_{\mathrm{alg}}(G)$ such that $\varphi(N)=N$ for some connected normal subgroup $N$. Let $\overline{\varphi}$ denote the automorphism of $G/N$ induced by $\varphi$. Then $R(\varphi)=1$ if and only if $R(\varphi|_N)=R(\overline{\varphi})=1$.
\end{theorem}

\begin{proof}
First, assume that $R(\varphi)=1$. Then obviously $R(\overline{\varphi})=1$ and hence, the fixed point subgroup $(G/N)^{\overline{\varphi}}$ is finite. Let $(G/N)^{\overline{\varphi}}=\{g_1N,\ldots, g_lN\}$ for some $g_1,\ldots,g_l\in G$, $l\in \mathbb{N}$. Note that for each $i\in \{1,\ldots,l\}$, $x_i:=g_i^{-1}\varphi(g_i)\in N$. We \textbf{claim} that $\{[x_i]_{\varphi|_N}:1\leq i\leq l\}$ is the set of all orbits of the $\varphi|_N$-twisted action of $N$ on itself. If the claim is true, then $R(\varphi|_N)<\infty$. Therefore by Lemma \ref{nonsolv}, $N$ is solvable and hence, $R(\varphi|_N)=1$ by Theorem \ref{solv}. To prove the claim, let $x\in N$ be an arbitrary element. Write $x=g^{-1}\varphi(g)$ for some $g\in G$ (this is possible as $R(\varphi)=1$). Thus $gN\in (G/N)^{\overline{\varphi}} $ and therefore, $gN=g_iN$ for some $i\in\{1,\ldots,l\}$. So we have $g_i^{-1}g\in N$ and hence, $g_i^{-1}gx\varphi(g^{-1}g_i)=g_i^{-1}\varphi(g_i)=x_i$ shows that $x$ is $\varphi|_N$-conjugate to $x_i$. This proves the claim.

Conversely, let $g\in G$ be an arbitrary element. Since $R(\overline{\varphi})=1$, there exists $x\in G$ such that $xg\varphi(x^{-1})\in N$. Subsequently, since $R(\varphi|_N)=1$, there exists $n\in N$ such that $nxg\varphi(x^{-1}n^{-1})=e$, as desired.
\end{proof}

\subsection{Borel subgroups of semisimple algebraic groups}\label{bor}

In this section we establish the algebraic $R_\infty$-property of Borel subgroups of semisimple algebraic groups. First, we observe the following useful result.
\begin{lemma}\label{tga}
If $G\cong T\ltimes U$, where $T= \mathbb{G}_m^n$ ($n\geq 1$) and $U= \mathbb{G}_a $,  then the following are equivalent:
\begin{enumerate}
\item There exists $\varphi\in \mathrm{Aut}_{\mathrm{alg}}(G)$ such that $\varphi(T)=T$ and $R(\varphi|_T)=1$.
\item $G$ is the direct product of $T$ and $U$.
\item $G$ fails to have the algebraic $R_\infty$-property.
\end{enumerate}
\end{lemma}
\begin{proof}

 Let
 the action of $T$ on $U$ be given by $txt^{-1}=\alpha_t x$ for all $t\in T, x\in U$, where $t\mapsto \alpha_t$ is a homomorphism $T\rightarrow k^\times$.

 \emph{(1) $\Rightarrow$ (2)} Assume that there exists an automorphism $\varphi\in \mathrm{Aut}_{\mathrm{alg}}(G)$ such that $\varphi(T)=T$ and $R(\varphi|_T)=1$. Let $\beta\in k^\times$ be such that $\varphi|_U(x)=\beta x$ for all $x\in U$. Then for every $t\in T$ and $x\in U$, we have $$t\varphi(x)t^{-1}=\alpha_t(\beta x)=\beta (\alpha_t x)=\varphi(txt^{-1})=\varphi(t)\varphi(x)\varphi(t^{-1}).$$ Thus $t^{-1}\varphi(t)\varphi(x)\varphi(t^{-1})t=\varphi(x)$ for all $t\in T, x\in U$ which implies that every element of the $\varphi|_T$-conjugacy class of $e$ in $T$ commutes with every element of $U$. But then (since $R(\varphi|_T)=1$) $T$ centralizes $U$ and hence, $G$ is the direct product of $T$ and $U$. 
 
 \vspace*{2mm}

\emph{(2) $\Rightarrow$ (3)} Let $G=T\times U$. Consider $\varphi_1\in \mathrm{Aut}_{\mathrm{alg}}(T)$ and $\varphi_2\in \mathrm{Aut}_{\mathrm{alg}}(U)$ such that $R(\varphi_1)=R(\varphi_2)=1$. Such a $\varphi_1$ exists by Theorem \ref{torus}. Let $\varphi_2\in \mathrm{Aut}_{\mathrm{alg}}(U)$ be such that $\varphi_2(x)=\beta x$ for all $x\in U$, where $\beta\in k\setminus \{0,1\}$ is a fixed scalar. Then for any $y\in U$, the equality $y=\frac{y}{1-\beta}+\frac{-\beta y}{1-\beta}$, shows that $y$ is $\varphi_2$-conjugate to the identity element in $U$ and hence, $R(\varphi_2)=1$. Then one checks that $\varphi(tx)=\varphi_1(t)\varphi_2(x)$ for all $t\in T, x\in U$ defines an automorphism of $G$ and hence $R(\varphi)=1$ by Theorem \ref{solv} (Case 1).

\vspace*{2mm}

\emph{(3) $\Rightarrow$ (1)} If $G$ does not have the algebraic $R_\infty$-property, then by Theorem \ref{solv} there exists $\psi\in \mathrm{Aut}_{\mathrm{alg}}(G)$ such that $R(\psi)=1$. Let $g\in G$ be such that $\mathrm{Int}_g\psi(T)=T$ and set $\varphi=\mathrm{Int}_g\psi$. Then by Lemma \ref{inner} $R(\varphi)=R(\psi)=1$ and hence, from the  proof of Theorem \ref{solv}, it follows that $R(\varphi|_T)=1$.
\end{proof}

Now, let $G$ be a connected semisimple algebraic group, $T$ a maximal torus and $B$ a Borel subgroup containing $T$. Let $\Phi$ be the root system of $G$ associated to $T$, $\Delta$ the simple subsystem of $\Phi$ determined by $B$ and $\Gamma$ the group of all automorphisms of $\Phi$ stabilizing $\Delta$.

For every $\psi\in \mathrm{Aut}_{\mathrm{alg}}(B)$ there exists $b\in B$ such that $\mathrm{Int}_b\psi(T)=T$ (since $T$ is a maximal torus in $B$). Thus if $D'=\{\varphi\in \mathrm{Aut}_{\mathrm{alg}}(B): \varphi(T)=T\}$, then $\mathrm{Aut}_{\mathrm{alg}}(B)=\mathrm{Int}(B)D'$. Let $X(T)$ denote the character group of $T$, and $\mathrm{Aut}_{\mathbb{Z}}(X(T))$ denote the group of all automorphisms of $X(T)$. Observe that we have a homomorphism $D'\rightarrow \mathrm{Aut}_{\mathbb{Z}}(X(T))$ given by $\varphi\mapsto \overline{\varphi}$ for all $\varphi\in D'$, where $\overline{\varphi}(\alpha)(t)=\alpha\varphi^{-1}(t)$ for all $\alpha\in X(T), t\in T$. We \textbf{claim} that $\overline{\varphi}\in \Gamma$ for all $\varphi\in D'$. To see this, let for each root $\alpha\in \Phi$, $U_\alpha$ be the root subgroup associated to it. It is known that $U_\alpha$ can be characterised as the image of any injective homomorphism $\epsilon:\mathbb{G}_a\rightarrow G$ such that $t\epsilon(x)t^{-1}=\epsilon(\alpha(t)x)$, for all $t\in T,x\in \mathbb{G}_a$ (see \cite[Section 26.3]{hum}). So if $\beta\in \Phi^+$, and $\epsilon_\beta:\mathbb{G}_a\rightarrow U_\beta$ an associated isomorphism, then for every $t\in T, x\in \mathbb{G}_a$, we have $t\varphi(\epsilon_\beta(x)t^{-1}=\varphi\epsilon_\beta(\overline{\varphi}(\beta)(t)x)$. This shows that $\varphi(U_\beta)=U_{\overline{\varphi}(\beta)}\subset B$. Hence $\overline{\varphi}(\beta)\in \Phi^+$, whenever $\beta\in \Phi^+$. Now since $\overline{\varphi}(\Delta)$ is again a simple subsystem contained in $\Phi^+$ and $\Phi^+$ determines $\Delta$ uniquely, we conclude that $\overline{\varphi}(\Delta)=\Delta$. This proves the claim. 

Now, by the argument used in the proof of Theorem 27.4 in \cite{hum}, it follows that the kernel of the homomorphism $D'\rightarrow \Gamma$ is exactly equal to $D'\cap \mathrm{Int}(B)$. We summarize this discussion as 

\begin{lemma}\label{autborel1}
Let $B$ be a Borel subgroup of a connected semisimple algebraic group $G$.
%and $D$ the group of all automorphisms of $B$ which stabilize $T$ 
Then 
\begin{enumerate}
\item $\mathrm{Aut}_{\mathrm{alg}}(B)=\mathrm{Int}(B)D'$, 
\item the kernel of the map $D'\rightarrow \Gamma$ is equal to $\mathrm{Int}(B)\cap D'$.
\end{enumerate}
\end{lemma}

\begin{corollary}\label{autborel2}
Let $G,B,T$ and $D'$ be as above. Assume further that $G$ is of simply connected or adjoint type. Then for every $\varphi\in D'$ there exists $\psi\in \mathrm{Aut}_{\mathrm{alg}}(G)$ such that $\varphi=\psi|_B$. 
\end{corollary}

\begin{proof}
If $D=\{\varphi\in \mathrm{Aut}_{\mathrm{alg}}(G):\varphi(B)=B,\varphi(T)=T\}$, then the natural map $D\rightarrow \Gamma$ is onto (c.f. \cite[Theorem 25.16]{kmrt}). Let $\varphi\in D'$ be arbitrary and consider its image $\overline{\varphi}\in \Gamma$. If  $\rho\in D$ is a preimage of $\overline{\varphi}$,  then by Lemma \ref{autborel1} we conclude that there exists $b\in B$ such that $\rho|_B=\mathrm{Int}_b\varphi$. Setting $\psi=\mathrm{Int}_{b^{-1}}\rho$, we get $\psi|_B=\varphi$.

\end{proof}

%\begin{proposition}
%Let $G=G_1^{n_1}\times\hdots\times G_l^{n_l}$ where $G_i$'s are simple algebraic groups with $G_i\ncong G_j$ for all $i\neq j$. Then any Borel subgroup of $G$ has $R_\infty$-property.
%\end{proposition}
%\begin{lemma}
%Let $G$ be a simple algebraic group, $B$ a Borel subgroup containing a maximal torus $T$. 
%\end{lemma}
%\begin{proof}
%We fix a maximal torus $T_i\subset G_i$ and a Borel subgroup $B_i$ containing $T_i$ for every $1\leq i\leq                  l$. Consider the Borel subgroup $B=B_1^{n_1}\times\hdots\times B_l^{n_l}$ of $G$ containing the torus $T=T_1^{n_1}\times\hdots\times T_l^{n_l}$.
%\end{proof}
\begin{theorem}\label{borel}
If $B$ is a Borel subgroup of a connected semisimple algebraic group $G$, then $B$ has the algebraic $R_\infty$-property.
\end{theorem}
\begin{proof}
First we assume that $G$ is of adjoint type. Let $G=G_1^{n_1}\times\dots\times G_l^{n_l}$ where $G_i$'s are simple algebraic groups, with $G_i\ncong G_j$ for all $i\neq j$. We fix a maximal torus $T_i\subset G_i$ and a Borel subgroup $B_i$ containing $T_i$ for every $1\leq i\leq l$. Consider the Borel subgroup $B=B_1^{n_1}\times\dots\times B_l^{n_l}$ of $G$ containing the torus $T=T_1^{n_1}\times\dots\times T_l^{n_l}$. It suffices to prove the theorem for this chosen $B$ (by virtue of the conjugacy of all Borel subgroups in $G$). Let $D'$ and $\Gamma$ be as in Lemma \ref{autborel1} and note that in view of Lemma \ref{autborel1} and Lemma \ref{inner}, it is enough to prove that $R(\varphi)=\infty$ for all $\varphi\in D'$. So let $\varphi\in D'$ and by Corollary \ref{autborel2}, we find an element $\psi\in \mathrm{Aut}_{\mathrm{alg}}(G)$ such that $\psi|_B=\varphi$. Note that $\psi(G_i^{n_i})=G_i^{n_i}$ and hence $\psi(B_i^{n_i})=B_i^{n_i}$ and $\psi(T_i^{n_i})=T_i^{n_i}$ for all $1\leq i\leq l$. This in turn implies that $B_i^{n_i}$ and $T_i^{n_i}$ are invariant under $\varphi$ for all $1\leq i\le l$.

\vspace*{4mm}

\noindent\textbf{Claim:} The unipotent radical of $B$ contains a connected one dimensional subgroup which is invariant under $\mathrm{Int}_t\varphi$, for some suitable $t\in T$.

\vspace*{3mm}
\noindent\textbf{Proof of claim:} Fix an $i\in \{1,\ldots,l\}$. Let $\Phi_i$ (respectively, $\Delta_i$) be the root system (respectively, simple subsystem) of $G_i$ determined by $T_i$ (respectively, $B_i$). 

By virtue of Lemma \ref{autsimple}, we identify $(\mathrm{Aut}_{\mathrm{alg}}(G_i))^{n_i}$ and $S_{n_i}$ as subgroups of that $\mathrm{Aut}_{\mathrm{alg}}(G_i^{n_i})$, and write $(\mathrm{Aut}_{\mathrm{alg}}(G_i))^{n_i}=(\mathrm{Aut}_{\mathrm{alg}}(G_i))^{n_i}S_{n_i}$. If $f_1,\ldots,f_{n_i}\in \mathrm{Aut}_{\mathrm{alg}}(G_i)$ and $\sigma\in S_{n_i}$, then the action of $(f_1,\ldots,f_{n_i})\sigma$ on $G_i^{n_i}$ is given by $$((f_1,\ldots,f_{n_i})\sigma)(g_1,\ldots,g_{n_i})=(f_1(g_{\sigma^{-1}(1)}),\ldots,f_{n_i}(g_{\sigma^{-1}(n_i)})),$$ for all $(g_1,\ldots,g_{n_i})\in G_i^{n_i}$. So if $\psi_i=\psi|_{G_i^{n_i}}$, then there exists $\sigma_i\in S_{n_i}\subset \mathrm{Aut}_{\mathrm{alg}}(G_i^{n_i})$ such that $\psi_i\sigma_i=(\psi_{i1},\ldots,\psi_{in_i})$, where $\psi_{i1},\ldots,\psi_{in_i}\in \mathrm{Aut}_{\mathrm{alg}}(G_i)$, and since $\psi_i\sigma_i$ leaves $B_i^{n_i}$ and $T_i^{n_i}$ invariant, it follows that $\psi_{ij}(B_i)=B_i$ and $\psi_{ij}(T_i)=T_i$ for all $1\leq j\leq n_i$. Hence each of the automorphisms $\psi_{ij}$ is induced by an automorphism (say) $\gamma_j$ of $\Phi_i$, which maps $\Delta_i$ to itself. Let $\alpha_{i}$ be the highest root in $\Phi_i^+$ determined by $\Delta_i$. Then $\gamma_j(\alpha_{i})=\alpha_{i}$ for all $1\leq j\leq n_i$. Therefore, if $U_{\alpha_{i}}$ is the root subgroup of $G_i$ associated to $\alpha_{i}$, then $\psi_{ij}(U_{\alpha_{i}})=U_{\alpha_{i}}$ for all $1\leq j\leq n_i$. Consider an isomorphism $\epsilon_{\alpha_i}:\mathbb{G}_a\rightarrow U_{\alpha_i}$ such that $t\epsilon_{\alpha_i}(x)t^{-1}=\epsilon_{\alpha_i}(\alpha_i(t)x)$, for all $t\in T_i, x\in \mathbb{G}_a$ and for each $j=1,\ldots,n_i$, let $c_j\in k^\times$ be such that $\psi_{ij}(\epsilon_{\alpha_i}(x))=\epsilon_{\alpha_i}(c_jx)$, for all $x\in \mathbb{G}_a$. Also, for each $j\in \{1,\ldots,n_i\}$, we can find $t_{ij}\in T_i$ such that $\alpha_i(t_{ij})=c_j^{-1}$ and hence $\mathrm{Int}_{t_{ij}}\psi_{ij}$ is identity on $U_{\alpha_i}$.
Consider the homomorphism $\theta_i:\mathbb{G}_a\rightarrow G_i^{n_i}$ defined by $\theta_i(x)=(\epsilon_{\alpha_i}(x),\ldots,\epsilon_{\alpha_i}(x))$, for all $x\in \mathbb{G}_a$.

 Now define $\theta:\mathbb{G}_a\rightarrow G$ by setting $\theta(x)=(\theta_1(x),\ldots,\theta_l(x))$, for all $x\in \mathbb{G}_a$. Note that $\theta $ is an isomorphism onto its image. So, if $N:=\theta(\mathbb{G}_a)$, $t_i:=(t_{i1},\ldots,t_{in_i})\in T_i^{n_i}$ ($1\leq i\leq l$) and $t:=(t_1,\ldots,t_l)\in T$, then we check that $\mathrm{Int}_{t}\varphi=(\mathrm{Int}_{t}\psi)|_B$ is identity on $N$. This proves the claim.

\vspace*{3mm}

\noindent Now assume that $R(\mathrm{Int}_{t}\varphi)=1$. Then by Lemma \ref{solvtorus} and Lemma \ref{unitorus},  $R(\mathrm{Int}_{t}\varphi|_T)=1$ and by the above claim $\mathrm{Int}_{t}\varphi$ stabilizes $T\ltimes N$. Therefore by Lemma \ref{tga} $T$ centralizes $N$, a contradiction. Thus $R(\varphi)=R(\mathrm{Int}_{t}\varphi)=\infty$ and this proves the theorem when $G$ is of adjoint type.

For the general case, let $G$ be any semisimple algebraic group with a Borel subgroup $B$. Consider the semisimple group $G_{ad}$ (of adjoint type) isogenous to $G$, via the adjoint homomorphism $\mathrm{Ad}:G\rightarrow G_{\mathrm{ad}}$. Then we have an exact sequence 
\[\xymatrix{e\ar[r]&Z(G)\ar[r]&G\ar[r]^{\mathrm{Ad}}&G_{ad}\ar[r]&e}.\] 

Now $Z(G)=Z(B)$ and $B_{ad}=\mathrm{Ad}(B)$ is a Borel subgroup of $G_{ad}$. Hence we have an exact sequence \[\xymatrix{e\ar[r]&Z(B)\ar[r]&B\ar[r]^{\mathrm{Ad}}&B_{ad}\ar[r]&e}.\] Since  $B_{ad}$ has the algebraic $R_\infty$-property, and $Z(B)$ is invariant under every automorphism of $B$, by Lemma \ref{ses} we conclude that $B$ has the algebraic $R_\infty$-property.

This completes the proof.
\end{proof}
\subsection{Maximal unipotent subgroups of simple algebraic groups}\label{uni}
In this section assume that $k$ is an algebraically closed field of characteristic different from $2$ and $3$. Let $G$ be a simple algebraic group over $k$, $B$  a Borel subgroup of $G$ and $U$ the unipotent radical of $B$. Note that $U$ is a maximal unipotent subgroup of $G$. It has been observed in \cite[Proposition 21]{bb} that $U$ admits an automorphism $\varphi$ for which $R(\varphi)=1$. We proceed to deduce a necessary and sufficient condition for the $\varphi$-conjugacy action to be transitive. This in turn will give a characterization of $\varphi$ for which $R(\varphi)=\infty$ (by virtue of Lemma \ref{unitorus}).
From the works of Gibbs \cite{gibbs} and Fauntleroy \cite{faunt}, we know that $\mathrm{Aut}_{\mathrm{alg}}(U)$ is generated by the inner automorphisms along with four other types of automorphisms, which we briefly describe below.

We view the group $G$ as a Chevalley group of type $\Phi$ based on $k$. Let $\Delta:=\{\alpha_1,\ldots, \alpha_l\}$ be the set of simple roots (where $l$ is the rank of $G$) and $\Phi^+:=\{\alpha_1,\ldots, \alpha_N\}$ the set of all positive roots. Every $\alpha\in \Phi^+$ can be uniquely written as $\alpha=\sum_{i=1}^ln_i\alpha_i$, where $n_i\in \mathbb{N}\cup\{0\}$. The \emph{height} of $\alpha$ is defined as  $\mathrm{ht}(\alpha):=\sum_{i=1}^ln_i$. Let $\alpha_N$ denote the unique root of maximum height in $\Phi^+$. Assume that  of $\Phi^+$ is endowed with an ordering : $\alpha_1<\alpha_2<\cdots<\alpha_{N-2}<\alpha_{N-1} <\alpha_N$, where $\mathrm{ht}(\alpha_i)\leq \mathrm{ht}(\alpha_j)$ if $\alpha_i<\alpha_j$. 
 
 If $\Phi$ is of the type $A_l$ ($l\geq 2$), then there are exactly two simple roots $\alpha_1, \alpha_l$ such that $\alpha_N-\alpha_1, \alpha_N-\alpha_l\in \Phi^+$. In this case assume that $\alpha_{N-2}=\alpha_N-\alpha_l$ and $\alpha_{N-1}=\alpha_N-\alpha_1$.
 If $\Phi$ is \emph{not} of the type $A_l$, then there is a unique $\alpha_i\in \Delta$ such that $\alpha_N-\alpha_i\in \Phi^+$. 
 In this case we assume that $\alpha_{N-1}=\alpha_N-\alpha_i$. Furthermore, if $\Phi$ is of type $C_l$ ($l\geq 3$), then $ \alpha_N-\alpha_i$ and $\alpha_N-2\alpha_{i}$ are in $\Phi^+$. Hence in this case we assume that $\alpha_{N-2}=\alpha_N-2\alpha_{i}$ and $\alpha_{N-1}=\alpha_N-\alpha_i$
 
 Now let $U=\langle x_\alpha(t):\alpha\in \Phi^+, t\in k\rangle$. We consider the following automorphisms of $U$.

\noindent\textbf{\textit{Extremal automorphisms:}}
 For every $u\in k$ there exists an automorphism $\varphi_{u}:U\rightarrow U$ such that $\varphi_{u}(x_{\alpha_j}(t))=x_{\alpha_j}(t)$ for all $\alpha_j\neq\alpha_i$
% \; ($j\in\{1,2,\ldots, N\}$) 
 and 
 \begin{align}\label{exal}
 \varphi_{u}(x_{\alpha_i}(t))=x_{\alpha_i}(t)x_{\alpha_N-\alpha_i}(ut)x_{\alpha_N}(\lambda_i ut^2),
 \end{align}
 where $\alpha_i\in \Delta$ is such that $\alpha_N-\alpha_i\in \Phi^+$.
% \[\varphi_{u}(x_{\alpha_i}(t))= \left\{
% \begin{array}{ll}
% x_{\alpha_i}(t) & \text{if} \; \alpha_N-\alpha_i\notin \Phi^+, \\
% x_{\alpha_i}(t)x_{\alpha_N-\alpha_i}(ut)x_{\alpha_N}(\lambda_i ut^2) & \text{if}\; \alpha_N-\alpha_i\in \Phi^+.
% \end{array}\right. \] 
%More precisely, if $\Phi$ is of type $A_l$ ($l\geq 3$) then the extremal automorphisms of $U$ are defined as elements of the set $\{\varphi_u\varphi_{u'}\mid u,u'\in k\}$, where $\varphi_u,\varphi_{u'}$ are defined as in Equation \eqref{exal}  corresponding to the simple roots $\alpha_i$ (for $i=1,l$).
\vspace*{2mm}

\noindent Furthermore, if $\Phi$ is of type $C_l$, then for every $u'\in k$ there exists an automorphism $\psi_{u'}:U\rightarrow U$ such that $\psi_{u'}(x_{\alpha_j}(t))=x_{\alpha_j}(t)$ for all $\alpha_j\neq\alpha_i$ 
%($j\in\{1,2,\ldots, N\}$)\; 
and 
\begin{align}\label{exa2}
\psi_{u'}(x_{\alpha_i}(t))=x_{\alpha_i}(t)x_{\alpha_N-2\alpha_i}(u't)x_{\alpha_N-\alpha_i}(\mu_i u't^2)x_{\alpha_N}(\nu_i u't^3),
\end{align}
where $\alpha_i\in \Delta$ is such that $\alpha_N-\alpha_i\in \Phi^+$.
\vspace*{2mm}
 
% \[\psi_{u'}(x_{\alpha_i}(t))= \left\{
% \begin{array}{ll}
% x_{\alpha_i}(t) & \text{if} \; \alpha_N-\alpha_i\notin \Phi^+, \\
%  x_{\alpha_i}(t)x_{\alpha_N-2\alpha_i}(u't)x_{\alpha_N-\alpha_i}(\mu_i u't^2)x_{\alpha_N}(\nu_i u't^3)& \text{if}\; \alpha_N-\alpha_i\in \Phi^+.
% \end{array}\right. \] 
 In the above formulas \eqref{exal} and \eqref{exa2},
  $\lambda_i=\frac{1}{2}c_{11}, \mu_i=\frac{1}{2}c_{11}, \nu_i=\frac{1}{3}c_{12}$ and $c_{ij}\in\{\pm 1,\pm 2, \pm 3\}$. These integers are coming from the Chevalley's commutator formula (\ref{commutator}).
  If $\Phi$ is neither of type $A_l$ nor of type $C_l$, then the extremal automorphisms of $U$ are defined as the elements of the set $\{\varphi_u:u\in k\}$. This forms a subgroup (isomorphic to the additive group of $k$) of $\mathrm{Aut}_{\mathrm{alg}}(U)$. If $\Phi$ is of type either $A_l$ ($l\geq 3$) or $C_l$, then the subgroup of $\mathrm{Aut}_{\mathrm{alg}}(U)$ generated by the extremal automorphisms of $U$, is isomorphic to the direct product of two copies of the additive group of $k$.

  Next, we determine the action of an extremal automorphism on an element of $U$. So let $\prod\limits_{j=1}^Nx_{\alpha_j}(t_j)$ be any arbitrary element of $U$ ($t_j\in k$).
  
  If $\Phi$ is of type $A_l$ ($l\geq 3$), then we have 
  \begin{align*}
  \varphi_u\varphi_{u'}\left(\prod\limits_{j=1}^Nx_{\alpha_j}(t_j)\right)
  =& \left(\prod\limits_{j=1}^{N-3}x_{\alpha_j}(t_j)\right)
   x_{\alpha_{N-2}}(t_{N-2}+u't_l)\\
   & x_{\alpha_{N-1}}(t_{N-1}+ut_1) x_{\alpha_N}(t_N+u\lambda_1 t_1^2+u'\lambda_l t_l^2),
  \end{align*}
  for all $u,u'\in k$.
  
  If $\Phi$ is of type $C_l$, and $\alpha_i\in\Delta$ is the unique simple root such that $\alpha_N-\alpha_i\in \Phi^+$, then we have
  \begin{align*}
  \varphi_u\psi_{u'}\left(\prod\limits_{j=1}^Nx_{\alpha_j}(t_j)\right)
 =&\left(\prod\limits_{j=1}^{N-3}x_{\alpha_j}(t_j)\right)x_{\alpha_{N-2}}(t_{N-2}+u't_i)\\
 &x_{\alpha_{N-1}}(t_{N-1}+ut_i+u'\mu_i t_i^2) x_{\alpha_N}(t_N+u\lambda_it_i^2+u'\nu_i t_i^3),
  \end{align*}
   for any $u,u'\in k$.
  
  If $\Phi$ is not of type $A_l$ or $C_l$, and $\alpha_i\in\Delta$, the unique simple root such that $\alpha_N-\alpha_i\in \Phi^+$, then 
  \begin{align*}
  \varphi_u\left(\prod\limits_{j=1}^Nx_{\alpha_j}(t_j)\right)=\left(\prod\limits_{j=1}^{N-2}x_{\alpha_j}(t_j)\right) x_{\alpha_{N-1}}(t_{N-1}+ut_i) x_{\alpha_N}(t_N+u\lambda_i t_i^2),
  \end{align*}
  for any $u\in k$.
  
 %\begin{align*}
 %&\varphi_{\omega}(x_{\alpha_1}(t_1)\cdots x_{\alpha_i}(t_i)\cdots\alpha_l(t_l)\cdots x_{\alpha_N}(t_N))\\
 %&=x_{\alpha_1}(t_1)x_{\alpha_N-\alpha_1}(ut_1)x_{\alpha_N}(\lambda_1 ut_1^2)\cdots x_{\alpha_l}(t_l)x_{\alpha_N-\alpha_l}(ut_l)x_{\alpha_N}(\lambda_l ut_l^2)\cdots x_{\alpha_N}(t_N)\\
 %&=x_{\alpha_1}(t_1)\cdots x_{\alpha_{N-3}}(t_{N-3})
 %x_{\alpha_N-\alpha_l}(t_{N-2}+ut_l) x_{\alpha_N-\alpha_1}(t_{N-1}+ut_1) x_{\alpha_N}(t_N+u(\lambda_1 t_1^2+\lambda_l t_l^2)).
 %\end{align*}

%If  $\Phi$ is not of type $ A_l, C_l$. Since $\mathrm{ht}(\alpha_N)\geq 3$ then we have  
 %\begin{align*}
 %&\varphi_{\omega}(x_{\alpha_1}(t_1)\cdots x_{\alpha_i}(t_i)\cdots\alpha_l(t_l)\cdots x_{\alpha_N}(t_N))\\
 %&=x_{\alpha_1}(t_1)\cdots x_{\alpha_i}(t_i)x_{\alpha_N-\alpha_i}(ut_i)x_{\alpha_N}(\lambda_i ut_i^2)\cdots x_{\alpha_N}(t_N)\\
 %&=x_{\alpha_1}(t_1)\cdots x_{\alpha_{N-2}}(t_{N-2}) x_{\alpha_N-\alpha_i}(t_{N-1}+ut_i) x_{\alpha_N}(t_N+u\lambda_i t_i^2).
 %\end{align*}

 %or the root system of type $\Phi=C_l\; (l\geq 3)$, we get  
 %\begin{align*}
 %&\varphi_{\omega}(x_{\alpha_1}(t_1)\cdots x_{\alpha_i}(t_i)\cdots\alpha_l(t_l)\cdots x_{\alpha_N}(t_N))\\
 %&=x_{\alpha_1}(t_1)\cdots x_{\alpha_i}(t_i)x_{\alpha_N-2\alpha_i}(u't_i)x_{\alpha_N-\alpha_i}(ut_i+u'\mu_i t_i^2)x_{\alpha_N}(u\lambda_it_i^2+u'\nu_i t_i^3)\cdots x_{\alpha_N}(t_N)\\
 %&=x_{\alpha_1}(t_1)\cdots x_{\alpha_N-2\alpha_i}(t_{N-2}+u't_i)x_{\alpha_N-\alpha_i}(t_{N-1}+ut_i+u'\mu_i t_i^2) x_{\alpha_N}(t_N+u\lambda_it_i^2+u'\nu_i t_i^3).
 %\end{align*}
 
\noindent\textbf{\textit{Central automorphisms:}} Let  $g_1,\ldots,g_l$ be endomorphisms of the additive group of $k$. The map $\varphi_C:U\to U$ defined by \[\varphi_{C}\left(\prod\limits_{j=1}^Nx_{\alpha_j}(t_j)\right)=\left(\prod\limits_{j=1}^{N-1}x_{\alpha_j}(t_j)\right)x_{\alpha_N}(t_N+\sum_{j=1}^lg_j(t_j)),\] for all $t_j\in k$, is called a central automorphism of $U$.

\noindent\textbf{\textit{Graph automorphisms:}} Given any $\rho\in \Gamma$ (the group of Dynkin diagram symmetries), the graph automorphism of $U$ (associated to $\rho$) is defined as the map $\varphi_{\rho}:U\rightarrow U$ given by 
\[\varphi_{\rho}\left(\prod\limits_{j=1}^Nx_{\alpha_j}(t_j)\right)=\prod\limits_{j=1}^Nx_{\rho(\alpha_j)}(t_j),\] for all $t_j\in k$.

\noindent\textbf{\textit{Diagonal automorphisms:}}
Let $P:=\mathbb{Z}\langle\Phi\rangle$ be the root lattice and $\chi: P\to k^{\times}$ a character.  The diagonal automorphism $\varphi_\chi$ of $U$ is defined by \[\varphi_{\chi}\left(\prod\limits_{j=1}^Nx_{\alpha_j}(t_j)\right)=\prod\limits_{j=1}^Nx_{\alpha_j}(\chi(\alpha_j)t_j),\] for all $t_j\in k$.

We are now in a position to state an important result about automorphisms of $U$ due to Fauntleroy \cite[Theorem 2.8]{faunt} and Gibbs \cite[Theorem 6.2]{gibbs}. 
\begin{lemma}\label{gibbs}
If $\varphi\in \mathrm{Aut}_{\mathrm{alg}}(U)$, then $\varphi=\varphi_{\rho} \varphi_{\chi} \varphi_{\omega}\varphi_C\mathrm{Int}_g$, where $\mathrm{Int}_g$ is an inner automorphism defined by $g\in U$, $\varphi_C$ is a central automorphism, $\varphi_{\omega}$ is an extremal automorphism,  $\varphi_{\chi}$ is a diagonal automorphism and $\varphi_{\rho}$ is a graph automorphism. 
\end{lemma}
 Let $H$ be the subgroup of $\mathrm{Aut}_{\mathrm{alg}}(U)$, generated by the extremal, central, diagonal and graph automorphisms of $U$. If $\psi\in \mathrm{Aut}_{\mathrm{alg}}(U)$ and $\mathrm{Int}_g\psi\in H$ for some $g\in U$, then by virtue of Lemma \ref{inner}, $R(\psi)=1$ if and only if $R(\mathrm{Int}_g\psi)=1$. So, consider any automorphism $\varphi$ of $U$ of the form $\varphi=\varphi_{\rho} \varphi_{\chi} \varphi_{\omega}\varphi_C\in H$. For now let us assume that $\Phi$ is \textbf{not} of type $A_2$. We associate a matrix $M(U,\varphi)$ to the pair $(U,\varphi)$ in the following way:

%\begin{enumerate}[leftmargin=*]
%\item \textit{$\Phi$ is of type $A_2$:} Let $\varphi_\omega=\varphi_u$ for some $u\in k^\times$. Then  define  
%%$$M(U,\varphi)=\begin{pmatrix}
%%\chi(\alpha_1)-1&\chi(\alpha_1)u&\\
%%\chi(\alpha_2)u&\chi(\alpha_2)-1&\\
%%&&\chi(\alpha_3)-1
%%\end{pmatrix}
%% $$ and if $\rho\neq 1$ then define $$M(U,\varphi)=\begin{pmatrix}
%%\chi(\alpha_2)u-1&\chi(\alpha_2)&\\
%%\chi(\alpha_1)&\chi(\alpha_1)u-1&\\
%%&&\chi(\alpha_3)-1
%%\end{pmatrix}.$$
% \[M(U,\varphi)= \left\{
%\begin{array}{ll}
%\begin{pmatrix}
%\chi(\alpha_1)-1&\chi(\alpha_1)u&\\
%\chi(\alpha_2)u&\chi(\alpha_2)-1&\\
%&&\chi(\alpha_3)-1
%\end{pmatrix} & \text{ if } \rho=\Id, \\
%\\
%\begin{pmatrix}
%\chi(\alpha_2)u-1&\chi(\alpha_2)&\\
%\chi(\alpha_1)&\chi(\alpha_1)u-1&\\
%&&\chi(\alpha_3)-1
%\end{pmatrix} & \text{ if } \rho\neq \Id.
%\end{array}\right. \]
%\textit{$\Phi$ is not of type $A_2$:}
For $h\in\mathbb{N}$, let $\beta_1<\beta_2<\cdots<\beta_n$ be all positive roots of height $h$. Then note that $\rho$ stabilizes the subset $\{\beta_1,\ldots,\beta_n\}$ of $\Phi^+$. Since $\rho$ is determined by a permutation of $\{1,\ldots,n\}$, we denote this permutation also by $\rho$  and note that $\rho(\beta_i)=\beta_{\rho(i)}$ for all $1\leq i\leq n$. Now define the matrix
$M_h(U,\varphi)=\begin{pmatrix}
m_{ij}
\end{pmatrix}$, where the rows are described as follows: 
\begin{enumerate}
\item If $\rho(\beta_i)=\beta_i$, then $m_{ii}=\chi(\beta_i)-1$  and $m_{ij}=0$ for all $j\neq i$.
\item If $\rho(i)\neq i$, then $m_{ii}=-1$, $m_{i\rho^{-1}(i)}=\chi(\beta_{\rho^{-1}(i)})$ and $m_{ij}=0$ for all $j\neq i,\rho^{-1}(i)$.
\end{enumerate}

\vspace*{3mm}

\noindent So, if $1=h_1<h_2<\cdots<h_r$ are all possible heights of the elements of $\Phi^+$, then define the following block diagonal matrix: 

$$M(U,\varphi):=\begin{pmatrix}
\boxed{M_{h_1}(U,\varphi)}&0&0&\hdots&0\\

0&\boxed{M_{h_2}(U,\varphi)}&0&\hdots&0\\
\vdots&\vdots&\ddots&&\vdots\\
0&0&0&\hdots&\boxed{M_{h_r}(U,\varphi)}
\end{pmatrix}.$$
%\end{enumerate}

\noindent As an illustration, let us compute the matrix $M(U,\varphi)$ for some particular cases: 
\begin{example}
     If $\rho=1$, 
     %(and $\Phi$ is not of type $A_2$) 
     then we have 
$$M(U,\varphi)=\mathrm{diag}(
\chi(\alpha_1)-1,\cdots ,\chi(\alpha_N)-1).$$
\end{example}

\begin{example}
Let $\Phi$ be the root system of type $D_4$. Let $\Delta=\{\alpha_1,\alpha_2,\alpha_3,\alpha_4\}$. Note that  $|\Phi^+|=12=N$ and  $\alpha_N=\alpha_1+2\alpha_2+\alpha_3+\alpha_4$ is the unique root of maximum height. Now fix an ordering of the positive roots: $\alpha_1<\alpha_2<\cdots<\alpha_{12}$. Here $h_i=i$ for $1\leq i\leq 5$. Suppose that $\rho(\alpha_1)=\alpha_3,  \rho(\alpha_3)=\alpha_4, \rho(\alpha_4)=\alpha_1$ and $\rho(\alpha_2)=\alpha_2$. Then we get 
$$M(U,\varphi)=\mathrm{diag}(M_1(U,\varphi),M_2(U,\varphi),M_3(U,\varphi),M_4(U,\varphi),M_5(U,\varphi)),$$
%$$M(U,\varphi)=\begin{pmatrix}
%M_1(U,\varphi)&&&&\\
%&M_{2}(U,\varphi)&&&\\
%&&M_3(U, \varphi)&&\\
%&&&M_{4}(U,\varphi)&\\
%&&&&M_{5}(U,\varphi)
%\end{pmatrix},$$
where
\begin{enumerate}
\item[]  $M_1(U,\varphi)=\begin{pmatrix}-1&0&0&\chi(\alpha_4)\\
0&\chi(\alpha_2)-1&0&0\\\chi(\alpha_1)&0&-1&0\\0&0&\chi(\alpha_3)&-1\end{pmatrix}$, 

\item[] $M_2(U,\varphi)=\begin{pmatrix}
-1&0&\chi(\alpha_7)\\\chi(\alpha_5)&-1&0\\0&\chi(\alpha_6)&-1
\end{pmatrix},$

\item[] $M_3(U,\varphi)=\begin{pmatrix}
-1&\chi(\alpha_9)&0\\0&-1&\chi(\alpha_{10})\\\chi(\alpha_8)&0&-1
\end{pmatrix},$

\item[] $M_4(U,\varphi)=(\chi(\alpha_{11})-1),$ 
\item[] $ M_5(U,\varphi)=(\chi(\alpha_{12})-1)$.
\end{enumerate}

\end{example}

Again, let $1=h_1<h_2<\cdots<h_r$ be all possible heights of the elements of $\Phi^+$ and consider the sequence of subgroups $U=U_{h_1}>\cdots > U_{h_r}$, where $U_{h_i}=\{\prod\limits_{\alpha\in \Phi^+}x_\alpha(t):t\in k ~\mathrm{and} ~\mathrm{ht}(\alpha)<h_i\Rightarrow t=0 \}$. It is clear that for any $\varphi\in \mathrm{Aut}_{\mathrm{alg}}(U)$, $\varphi(U_{h_i})=U_{h_i}$. Let $\varphi_i:=\varphi|_{U_{h_i}}$ ($1\leq i\leq r$), $\overline{\varphi_i}$ the automorphism of $U_{h_i}/U_{h_{i+1}} $ induced by $\varphi_i$ ($1\leq i\leq r-1$) and set $\overline{\varphi_r}:=\varphi_r$. Then by virtue of Theorem \ref{general}, we have

\begin{lemma}\label{generalunip}
$R(\varphi)=1$ if and only if $R(\overline{\varphi_i})=1$ for all $1\leq i\leq r$.

\end{lemma}
With the above preparation we proceed to prove the following:

\begin{theorem}\label{maxunipotent}
Let $U$ be a maximal unipotent subgroup of a simple algebraic group $G$  and consider any automorphism $\psi\in \mathrm{Aut}_{\mathrm{alg}}(U)$. Let $y\in U$ be such that $\psi\mathrm{Int}_y=\varphi=\varphi_{\rho} \varphi_{\chi} \varphi_{\omega}\varphi_C$. Also assume that the root system $\Phi$ of $G$ is  not of type $A_2$.  Then the following are equivalent:
\begin{enumerate}
\item  $R(\psi)=1$.
\item  $R(\varphi)=1$.
\item The matrix $M(U,\varphi)$ is invertible.
\end{enumerate}
\end{theorem}

\begin{proof}
\emph{(1) $\Leftrightarrow$ (2)} This is a consequence of Lemma \ref{inner}, since $\psi$ and $\varphi$ differ by an inner conjugation.
\vspace*{2mm}

\emph{(2) $\Leftrightarrow$ (3)} In view of Lemma \ref{generalunip}, it suffices to show that $R(\overline{\varphi_i})=1$ if and only if $M_{h_i}(U,\varphi)$ is invertible for all $1\leq i\leq r$. To see this, first observe that $U_{h_i}/U_{h_{i+1}}\cong \mathbb{G}_a^{l_i}$, where $l_i$ is the number of positive roots of height $h_i$. Therefore, it only remains to be shown that the automorphism $\overline{\varphi_i}$ is given by the matrix $(M_{h_i}(U,\varphi)+\mathrm{Id})$ (c.f. Example \ref{ex2}) and this is achieved via the following computations:
 
 Let $x=x_{\alpha_1}(t_1)\cdots x_{\alpha_N}(t_N)$ ($t_i\in k$) be an arbitrary element in $U$. Consider the following cases:
\begin{enumerate}
%\item\emph{$\Phi$ is of type $A_2:$} Here $\Phi^+=\{\alpha_1,\alpha_2,\alpha_3\}$ with  $\alpha_3=\alpha_1+\alpha_2$. We have 
%\begin{align*}
%&\varphi(x_{\alpha_1}(t_1)x_{\alpha_2}(t_2)x_{\alpha_3}(t_3))=\varphi_{\rho} \varphi_{\chi} \varphi_{\omega}\varphi_C(x_{\alpha_1}(t_1)x_{\alpha_2}(t_2)x_{\alpha_3}(t_3))=\\
%&x_{\alpha_1}(\chi(\alpha_2)(t_2+ut_1))x_{\alpha_2}(\chi(\alpha_1)(t_1+ut_2))x_{\alpha_3}(\chi(\alpha_3)(t_3+g_1(t_1)+g_2(t_2)\\
%&+u(\lambda_1t_1^2+\lambda_2t_2^2)+2\lambda_1(u^2t_1t_2+ut_2^2))+2\lambda_1\chi(\alpha_1)\chi(\alpha_2)(t_1+ut_2)(t_2+ut_1))\\
%&=x_{\alpha_1}(\chi(\alpha_2)(t_2+ut_1))x_{\alpha_2}(\chi(\alpha_1)(t_1+ut_2))x_{\alpha_3}(\chi(\alpha_3)t_3+f(t_1,t_2), \end{align*}
%where $f(X,Y)$ is a polynomial (over $k$) with zero constant term.
\item\label{al}\emph{$\Phi$ is of type $A_l\;(l\geq 3):$}
\begin{align*}
&\varphi\left(\prod_{j=1}^Nx_{\alpha_j}(t_j)\right)=\varphi_{\rho} \varphi_{\chi} \varphi_{\omega}\varphi_C\left(\prod_{j=1}^Nx_{\alpha_j}(t_j)\right)\\
=&\left(\prod\limits_{j=1}^{N-3}x_{\rho(\alpha_j)}(\chi(\alpha_j)t_j)\right)
x_{\alpha_{N-2}}(\chi(\alpha_{N-1})(t_{N-1}+ut_1))\\ &x_{\alpha_{N-1}}(\chi(\alpha_{N-2})(t_{N-2}+u't_l))x_{\alpha_N}(\chi(\alpha_{N})(t_N+\sum_{j=1}^lg_j(t_j)+u\lambda_1 t_1^2+u'\lambda_l t_l^2)).
\end{align*} 
\item\label{cl}\emph{$\Phi$ is of type $C_l:$}
\begin{align*}
&\varphi\left(\prod_{j=1}^Nx_{\alpha_j}(t_j)\right)=\varphi_{\chi} \varphi_{\omega}\varphi_C\left(\prod_{j=1}^Nx_{\alpha_j}(t_j)\right)\\
=&\left(\prod_{j=1}^{N-3}x_{\alpha_j}(\chi(\alpha_j)t_j)\right) x_{\alpha_{N-2}}(\chi(\alpha_{\alpha_{N-2}})(t_{N-2}+u't_i))\\
&x_{\alpha_{N-1}}(\chi(\alpha_{N-1})(t_{N-1}+ut_i+u'\mu_i t_i^2) x_{\alpha_N}(\chi(\alpha_{N})(t_N+\sum_{j=1}^lg_j(t_j)+u\lambda_it_i^2+u'\nu_i t_i^3)).
\end{align*}
\item\label{alcl}\emph{$\Phi$ is not of type $A_l$ or $C_l:$}
\begin{align*}
&\varphi\left(\prod_{j=1}^Nx_{\alpha_j}(t_j)\right)=\varphi_{\rho} \varphi_{\chi} \varphi_{\omega}\varphi_C\left(\prod_{j=1}^Nx_{\alpha_j}(t_j)\right)=\left(\prod\limits_{j=1}^{N-2}x_{\rho(\alpha_j)}(\chi(\alpha_j)t_j)\right)\\
 &x_{\alpha_{N-1}}(\chi(\alpha_{N-1})(t_{N-1}+ut_i))x_{\alpha_N}(\chi(\alpha_{N})(t_N+\sum_{j=1}^lg_j (t_j)+u\lambda_i t_i^2)).
\end{align*}
\end{enumerate}

So, if we start with an arbitrary element in $U_{h_i}$, and read the equations in \eqref{al},\eqref{cl} and \eqref{alcl} modulo $U_{h_{i+1}}$, then it is clear that $\overline{\varphi_i}$ is described by the matrix $M_{h_i}(U,\varphi)+\mathrm{Id}$ ($1\leq i\leq r-1$). For $i=r$, note that $\varphi_r(x_{\alpha_N}(t))=x_{\alpha_N}(\chi(\alpha_N)t)$, for all $t\in k$, thereby showing that $\overline{\varphi_r}(=\varphi_r)$ is given  by multiplication by $\chi(\alpha_N)$. This completes the proof.
\end{proof}

As an immediate consequence of the above theorem, we record the following

\begin{corollary}
 (a) If $\varphi_\chi=\mathrm{Id}$, then $R(\varphi)=\infty$.
 
 \noindent (b) For any extremal automorphism $\varphi_\omega$, we have $R(\varphi_\omega)=\infty$.
 
 \noindent
 (c) If $\varphi_\rho=\mathrm{Id}$, then $R(\varphi)=\infty$ if and only if $\chi(\alpha_i)=1$ for some $i=1,\ldots,N$.
\end{corollary}
\begin{proof}
 Since the last block $M_{h_r}(U,\varphi)$ appearing in the matrix $M(U,\varphi)$ is a $1\times 1$ matrix given by $(\chi(\alpha_N)-1)$, (a) follows. Part (b) follows from the fact that an extremal automorphism acts trivially on the root subgroup $U_{h_r}=\{x_{\alpha_N}(t): t\in k\}$. For part (c), it suffices to observe that $M(U,\varphi)$ in this case, is equal to $\mathrm{diag}(\chi(\alpha_1)-1,\ldots,\chi(\alpha_N)-1)$.
\end{proof}
\begin{remark}
The reason for not including the root system of type $A_2$ in Theorem \ref{maxunipotent} is that the matrix $M(U,\varphi)$ looks quite different in general. We compute this matrix for a particular case in Example \eqref{a2} below.
\end{remark}

\subsection{Examples}\label{ex}

\begin{enumerate}

\item Tori do not have the algebraic $R_\infty$-property (c.f. Theorem \ref{torus}).

\vspace*{2mm}

\item It follows from Theorem \ref{solv} that  a connected nilpotent algebraic group $G$ has the algebraic $R_\infty$-property if and only if its unipotent radical has the algebraic $R_\infty$-property.

\vspace*{2mm}

\item\label{ex2} Let $G=\mathbb{G}_a^n$. We write each element of $G$ as an $n\times 1$ column vector in $k^n$. Then $\mathrm{GL}_n(k)$ is naturally identified with a subgroup of $\mathrm{Aut}_{\mathrm{alg}}(G)$ via $A\mapsto \varphi_A$ for all $A\in \mathrm{GL}_n(k)$, where $\varphi_A$ is the automorphism of $G$ given by $\overline{x}\mapsto A\overline{x}$, for all $\overline{x}\in G$.
 One checks that for any $\varphi\in \mathrm{GL}_n(k)\subset \mathrm{Aut}_{\mathrm{alg}}(G)$, $R(\varphi)=1$ if and  only if $\det(\varphi-\mathrm{Id})\neq 0$. 

\vspace*{2mm}

\item\label{ex3} For $n\geq 1$ and $r\geq 2$, let $\theta_1,\theta_2$ be homomorphisms of $\mathbb{G}_m^n\rightarrow \mathrm{GL}_r(k)$ defined by \begin{align*}
\theta_1(t_1, \ldots, t_n)&=\begin{pmatrix}
\boxed{\begin{matrix}t_1&0\\
0&t_1^{-1}\end{matrix}}&0\\
0&\boxed{I_{r-2}}
\end{pmatrix},\\  
\theta_2(t_1,\ldots, t_n)&=\begin{pmatrix}
t_1&0\\
0&\boxed{I_{r-1}}
\end{pmatrix},
\end{align*} for all $t_i\in \mathbb{G}_m$. Identifying $\mathrm{GL}_r(k)$ with a subgroup of $\mathrm{Aut}_{\mathrm{alg}}(\mathbb{G}_a^r)$, we consider the semidirect products $G_i=\mathbb{G}_m^n\ltimes_{\theta_i} \mathbb{G}_a^r$ ($i=1,2$).

\vspace*{2mm}

\noindent\textbf{(a)} First, we observe that $G_1$ does not have the algebraic $R_\infty$-property. Fix $a,b\in k^\times$, $B\in \mathrm{GL}_{r-2}(k)$ such that $ab\neq 1$ and $(B-I_{r-2})\in \mathrm{GL}_{r-2}(k)$.
Let $\varphi_1:\mathbb{G}_m^n\rightarrow\mathbb{G}_m^n$ and $\varphi_2:\mathbb{G}_a^r\rightarrow\mathbb{G}_a^r$ be the automorphisms given by $\varphi_1(t_1,\ldots,t_n)=(t_1^{-1},\ldots,t_n^{-1})$ and $\varphi_2(\overline{x})=\begin{pmatrix}
\boxed{\begin{matrix}0&a\\
b&0\\\end{matrix}}&0\\
0&\boxed{B}
\end{pmatrix}\overline{x}$, for all $t_i\in \mathbb{G}_m, \overline{x}\in \mathbb{G}_a^r$. Then by  Theorem \ref{torus} and Example \eqref{ex2} above, $R(\varphi_1)=R(\varphi_2)=1$. A direct calculation shows that $\varphi((t_1,\ldots,t_n),\overline{x}):=(\varphi_1(t_1,\ldots,t_n),\varphi_2(\overline{x}))$, for all $t_i\in \mathbb{G}_m, \overline{x}\in \mathbb{G}_a^r$, defines an automorphism of $G_1$. Therefore by Theorem \ref{solv} (Case 1), we  conclude that $R(\varphi)=1$.

\vspace*{2mm}

\noindent\textbf{(b)} The group  $G_2$ has the  algebraic $R_\infty$-property. Let $\psi\in \mathrm{Aut}_{\mathrm{alg}}(G_2)$ be any automorphism of $G_2$. Then for a suitable $g\in G_2$, the automorphism $\varphi=\mathrm{Int}_g\psi$ maps $\mathbb{G}_m^n$ onto itself. By virtue of Lemma \ref{inner}, it suffices to show that $R(\varphi)=\infty$. This follows from the \textbf{claim} that $R(\varphi|_{\mathbb{G}_m^n})=\infty$ (by Theorem \ref{solv} (Case 3)).  Before we prove the claim let us recall a definition. Let $\mathrm{char}(k)=p$. A \textit{$p$-polynomial in one variable} is defined as a polynomial of the form $f(X)=\sum\limits_{i=0}^na_iX^{p^i}$ for some positive integer $n$ and scalars $a_0,\ldots,a_n\in k$, if $p>0$ (respectively, $f(X)=aX$ for some $a\in k$, if $p=0$). 

\vspace*{2mm}

\noindent\textbf{Proof of claim}: Let $\varphi|_{\mathbb{G}_m^n}=\varphi_1$ and $\varphi|_{\mathbb{G}_a^r}=\varphi_2$. If possible let $R(\varphi_1)=1$. Then by Theorem \ref{torus}, $\varphi_1=(a_{ij})\in \mathrm{GL}_n(\mathbb{Z})$ such that $\mathrm{det}((a_{ij})-\mathrm{Id})\neq 0$; and $\varphi_2$ is given by :
 For every $\overline{x}=(x_1,\ldots,x_r)\in \mathbb{G}_a^r$, $\varphi_2(\overline{x})=\overline{y}$, where the $j^\mathrm{th}$ coordinate of the vector $\overline{y}$ is $y_j=\sum \limits_{i=1}^r f_{ji}(x_i)$, each $f_{ji}$ being a $p$-polynomial in one variable over $k$ (c.f. \cite{rosen}). Now, for every $(t_1,\ldots,t_n)\in \mathbb{G}_m^n$ we have $\varphi_2\theta_2((t_1,\ldots,t_n))=\theta_2(\varphi_1((t_1,\ldots,t_n)))\varphi_2$; evaluating on an arbitrary element $\overline{x}\in \mathbb{G}_a^r$, we obtain
\begin{equation*}
f_{j1}(t_1x_1)=f_{j1}(x_1) ~(2\leq j\leq r),
\end{equation*}
and
\begin{equation*}
\begin{split}
&f_{11}(t_1x_1)+f_{12}(x_2)+\cdots +f_{1r}(x_r)\\
=&~t_1^{a_{11}}t_2^{a_{12}}\cdots t_n^{a_{1n}}(f_{11}(x_1)+\cdots+f_{1r}(x_r)).
\end{split}
\end{equation*}
%$f_{11}(t_1x_1)+f_{12}(x_2)+\cdots +f_{1r}(x_r)=t_1^{a_{11}}t_2^{a_{12}}\cdots t_n^{a_{1n}}(f_{11}(x_1)+\cdots+f_{1r}(x_r))$ and $f_{j1}(t_1x_1)=f_{j1}(x_1)$ (for all $2\leq j\leq r$). 
Since $t_i$'s and $\overline{x}$ are arbitrary and $(a_{ij})$ is invertible, we conclude that $f_{j1}=0=f_{1j}$ (for all $2\leq j\leq r$) and $f_{11}(t_1x_1)=t_1^{a_{11}}t_2^{a_{12}}\cdots t_n^{a_{1n}}f_{11}(x_1)$. This shows that $\varphi_2$ maps the subgroup $\mathbb{G}_a\times 0\times\cdots\times 0$ isomorphically onto itself via $f_{11}$, thereby implying that $f_{11}(X)=cX$ for some $c\in k$ and a variable $X$. Thus we have  $ct_1=ct_1^{a_{11}}t_2^{a_{12}}\cdots t_n^{a_{1n}}$, for all $t_i\in k^\times$. Now suppose that $a_{11}\neq 1$. Then by taking $t_2=\cdots=t_n=1$ and $t_1$ to be such that $t_1^{a_{11}-1}\neq 1$, we note that $c=0$. On the other hand, if $a_{11}=1$, then at least one of $a_{12},\ldots,a_{1n}$ is non-zero; for if it is not the case, then the first row of the matrix $(a_{ij})-\mathrm{Id}$ becomes zero, contrary to the fact that $\mathrm{det}((a_{ij})-\mathrm{Id})\neq 0$. Thus by taking suitable values for $t_2,\ldots, t_n$, we again infer that $c=0$. Hence $f_{11}=0$, a final contradiction. \qed

\textbf{Alternatively}, we can show that $R(\varphi)=\infty$ via the following argument:

 First, note that $\varphi$ stabilizes the commutator subgroup $U:=[\mathbb{G}_m^n,G_2]$. One checks that $U$ is isomorphic to $\mathbb{G}_a$ and $\mathbb{G}_m^n$ acts nontrivially on $U$ via conjugation. Therefore, $\varphi$ restricts to an automorphism (say) $\psi$ of $\mathbb{G}_m^n\ltimes U$. Now if possible let  $R(\varphi)=1$. Then by Theorem \ref{solv}, $R(\varphi|_{\mathbb{G}_m^n})=R(\psi|_{\mathbb{G}_m^n})=1$. But then by Lemma \ref{tga}, $\mathbb{G}_m^n$ centralizes $U$, a contradiction.

%\item Let $char(k)=p>0$ and consider the rank $2$ Witt group $W_2(k)$. This is a two dimensional connected unipotent algebraic group whose underlying variety is $\mathbb{A}_k^2$ and multiplication is defined by:
%$$(a_0,a_1)\ast (b_0,b_1):=(a_0+b_0,a_1+b_1-f(a_0,b_0)) ~\forall a_i,b_i\in k,$$ where $f\in k[X,Y]$ is given by $f(X,Y)=\sum\limits_{i=1}^{p-1}c_iX^iY^{p-i}$ with $c_i=\frac{(p-1)!}{i!(p-i)!}$.
%Fix $\alpha,\gamma\in k^\times$ with $\alpha^p\neq 1$ and consider the automorphism $\varphi:W_2(k)\rightarrow W_2(k)$ defined by $\varphi((a_0,a_1))=(\alpha a_0,\gamma a_0+\alpha^pa_1)$ for all $a_i\in k$. Then one checks that $R(\varphi)=1$.

\vspace*{2mm}

\item Let $\mathrm{char}(k)=p>0$. The set $k^n$ can be endowed with the structure of a ring with identity via a construction due to Witt. Denote this ring by $(W_n(k),\oplus,\circ,0,1)$, where the roles of $0$ and $1$ are played by the elements $(0,\ldots,0)$ and $(1,\ldots,0)$ respectively (c.f. \cite{jacobson2} for all relevant definitions). The group $G=(W_n(k),\oplus,0)$ is a connected unipotent commutative algebraic group whose underlying affine variety is given by $\mathbb{A}_k^n$. It can be shown that an element $(\lambda_0,\lambda_1,\cdots, \lambda_{n-1})$ ($\lambda_i\in k$) admits a multiplicative inverse in the ring $W_n(k)$ if and only if $\lambda_0\neq 0$. So let $\lambda=(\lambda_0,0,\ldots,0)\in W_n(k)$ with $\lambda_0\neq 0$ and $\lambda_0^{p^i}\neq 1$ (for $0\leq i\leq n-1$). Note that the left homothety $\varphi_\lambda$ (defined by $\lambda$) gives an automorphism of $G$ (c.f. \cite[Lemma 3.3]{proud}). We check that the fixed point subgroup $G^{\varphi_\lambda}$ is trivial. Indeed, for if $x=(x_0,\ldots,x_{n-1})\in G^{\varphi_\lambda}$, then $(x_0,\ldots,x_{n-1})=\varphi_\lambda((x_0,\cdots,x_{n-1}))=(\lambda_0x_0,\lambda_0^px_1,\ldots,\lambda_0^{p^{n-1}}x_{n-1})$ implies that $x_i=0$, for all $0\leq i\leq n-1$. Hence by Lemma \ref{stalg} $R(\varphi_\lambda)=1$.

\item\label{a2} Let $G$ be a simple algebraic group with root system $\Phi$ of type $A_2$. Let $\Phi^+=\{\alpha_1,\alpha_2, \alpha_3=\alpha_1+\alpha_2\}$ and  $U$ be the maximal unipotent subgroup of $G$ generated by $\{x_{\alpha_i}(t_i):t_i\in k,~i=1,2,3\}$. Note that $U=U_1> U_2> 1$ is both the lower and the upper central series for $U$, where $U_2=\langle x_{\alpha_3}(t): t\in k\rangle\cong \mathrm{G}_a$ and  $U/U_2=\langle\overline{x_{\alpha_1}(t)},\overline{x_{\alpha_2}(s)}: t,s\in k\rangle\cong\mathbb{G}_{a}^2$. The group of all extremal automorphisms of $U$ is the subgroup of $\mathrm{Aut}_{\mathrm{alg}}(U)$ generated by the subset $\{\varphi_u,\psi_{u'}: u,u'\in k\}$, where $\varphi_u$ and $\psi_{u'}$ are defined by:
\begin{align*}
\varphi_u(x_{\alpha_1}(t))=x_{\alpha_1}(t)x_{\alpha_2}(ut)x_{\alpha_3}(\lambda_1ut^2);\\\varphi_u(x_{\alpha_j}(t))=x_{\alpha_j}(t),~ (j=2,3),~\mathrm{for}~\mathrm{all}~ t\in k,
\end{align*} and  
\begin{align*}
\psi_{u'}(x_{\alpha_2}(t))=x_{\alpha_2}(t)x_{\alpha_1}(u't)x_{\alpha_3}(\lambda_2u't^2);\\
  \psi_{u'}(x_{\alpha_j}(t))=x_{\alpha_j}(t),~ (j=1,3),~\mathrm{for}~\mathrm{all}~ t\in k,
  \end{align*} for some $\lambda_1,\lambda_2\in k$ (c.f. formula \eqref{exal}). For $u_1,u_1',u_2,u_2'\in k$, consider the extremal automorphism $\varphi_\omega:=\varphi_{u_1}\psi_{u_1'}\varphi_{u_2}\psi_{u_2'}$. Let  $\varphi:=\varphi_{\rho}\varphi_{\chi}\varphi_\omega\varphi_C$, where $\varphi_\rho$ is the graph automorphism induced by $\alpha_1\mapsto \alpha_2; \alpha_2\mapsto \alpha_1$, $\varphi_\chi$ is the diagonal automorphism defined by a character $\chi$, $\varphi_C$ is a central automorphism. Let $\overline{\varphi}$ denote the automorphism of $U/U_2$ induced by $\varphi$.
Now for any $t,s\in k$, we get,
\begin{equation*}
\begin{split}
&\overline{\varphi}(\overline{x_{\alpha_1}(t)})\\
=&~\overline{x_{\alpha_2}(\chi(\alpha_1)(1+u_1'u_2)t)}
\overline{x_{\alpha_1}(\chi(\alpha_2)(u_1+u_2+u_1u_1'u_2)t)}\\
&\mathrm{and,}\\
&\overline{\varphi}(\overline{x_{\alpha_2}(s)})\\
=&~\overline{x_{\alpha_2}(\chi(\alpha_1)(u_1'+u_2'+
	u_1'u_2'u_2)s)}
\overline{x_{\alpha_1}(\chi(\alpha_2)(1+u_1u_1'+u_2u_2'+u_1u_2'+u_1u_1'u_2u_2')s)}.
\end{split}
\end{equation*}

Thus the automorphism $\overline{\varphi}$ (viewed as an automorphism of $\mathbb{G}_a^2$) is given by the matrix 
\[ M:=\begin{pmatrix}\chi(\alpha_2)(u_1+u_2+u_1u_1'u_2)&\chi(\alpha_2)(1+u_1u_1'+u_2u_2'+u_1u_2'+u_1u_1'u_2u_2')\\\chi(\alpha_1)(1+u_1'u_2)&\chi(\alpha_1)(u_1'+u_2'+u_1'u_2'u_2)
\end{pmatrix}.\] Therefore by Lemma \ref{generalunip}, $R(\varphi)=1$ if and only if $R(\overline{\varphi})=1$ and $R(\varphi|_{U_2})=1$. Now $R(\overline{\varphi})=1$ if and only if $\det(M-  \mathrm{Id})\neq 0$, and $R(\varphi|_{U_2})=1$ if and only if $\chi(\alpha_3)\neq 1$. In general, a similar computation can be carried out for an arbitrary automorphism of $U$.
\end{enumerate}

\section{A comment on abstract $R_\infty$-property} \label{abst}

 Let $G$ be a connected semisimple algebraic group over $k$. If $\mathrm{char}(k)>0$, then for any Frobenius  automorphism $\sigma$ of $G$, the group $G^\sigma$ is finite. Since $\sigma$ is a surjective homomorphism of algebraic groups, by Lemma \ref{stalg}, $R(\sigma)=1$. However by Lemma \ref{nonsolv}, we know that for every algebraic group automorphism $\psi$ of $G$, $R(\psi)=\infty$. In \cite[Theorem 4.1]{FN} it has been shown that if $\mathrm{char}(k)=0$ and $\mathrm{tr.deg}_\mathbb{Q}k<\infty$, then for every abstract automorphism $\theta$ of $G$, $R(\theta)=\infty$. On the other hand it is known that if $\mathrm{tr.deg}_\mathbb{Q}k$ is infinite, and $G$  is one of the groups $\mathrm{GL}_n(k)$ \cite[Theorem 7]{timurjaa}, $\mathrm{SO}_n(k)$ and $\mathrm{Sp}_{2n}(k)$ \cite[Corollary 1, Theorem 6]{nas20} ($n\geq 1$), then  there exists an abstract automorphism (say) $\varphi$ of $G$ such that $R(\varphi)=1$. Following a line of argument as in the proof of \cite[Theorem 6]{timurjaa}, we now proceed to deduce an analogue of this result for Borel subgroups of simple algebraic groups.

\begin{lemma}\label{poly}
Let $A$ be a commutative ring with $1$ and $\sigma$ an automorphism of $A$. Suppose that $\begin{pmatrix}
a_{ij}
\end{pmatrix}\in \mathrm{GL}_n(A)$ and let $b_1,\ldots,b_n\in A$. Then $\sigma$ extends to an automorphism $\widetilde{\sigma}$ of $A[X_1,\ldots,X_n]$ such that $\widetilde{\sigma}(X_i)=\sum\limits_{j=1}^na_{ij}X_j+b_i$ for all $1\leq i\leq n$.
\end{lemma}

\begin{proof}
Consider the $A$-algebra structures defined on $A[X_1,\ldots,X_n]$ defined via the homomorphisms $f_1, f_2 (:A\rightarrow A[X_1,\ldots,X_n])$ where $f_1(a)=a$ and $f_2(a)=\sigma (a)$, for all $a\in A$. Then by the universal mapping property, there exists a ring homomorphism $\widetilde{\sigma}: A[X_1,\ldots,X_n]\rightarrow A[X_1,\ldots,X_n]$ such that $\widetilde{\sigma}|_A=\sigma$ and $\widetilde{\sigma}(X_i)=\sum\limits_{j=1}^na_{ij}X_j+b_i$ for all $1\leq i\leq n$. By a similar argument, one obtains a ring homomorphism $\rho$ of $A[X_1,\ldots,X_n]$ to itself such that $\rho|_A=\sigma^{-1}$ and $\rho(X_i)=\sum\limits_{j=1}^nc_{ij}(X_j-\sigma^{-1}(b_j))$ ($1\leq i\leq n$), where $\begin{pmatrix}
c_{ij}
\end{pmatrix}=\begin{pmatrix}
\sigma^{-1}(a_{ij})
\end{pmatrix}^{-1}$. We observe that $\rho$ and $\widetilde{\sigma}$ are inverses of one another and this proves the lemma.
\end{proof}

Let $k$ be an algebraically closed field of countable transcendence degree over $\mathbb{Q}$. Without loss of generality we assume that  $k=\overline{\mathbb{Q}(X_i:i\in \mathbb{N})}$.
Let $G$, $\Phi$ and $\Delta$ be as in Section \ref{prel} with the root system $\Phi$ being assumed to be irreducible (of rank $l$). Let $\Delta=\{\alpha_1,\ldots,\alpha_l\}$ and $\Phi^+=\{\alpha_1,\ldots,\alpha_N\}$. For any subfield $F$ of $k$ let $B(F)$ denote the subgroup of $G$ generated by $\{h_\alpha(t):\alpha\in \Delta,t\in F^\times\}\cup\{x_{\alpha}(s):\alpha\in \Phi^+,s\in F\}$ and set $B(k)=B$. Every automorphism $\psi$ of $k$ induces an abstract automorphism of $\widetilde{\psi}:B\rightarrow B$ such that $\widetilde{\psi}(h_\alpha(t))=h_\alpha(\psi(t))$ and $\widetilde{\psi}(x_\beta(s))=x_\beta(\psi(s))$, for all $\alpha\in \Delta,\beta\in \Phi^+,t\in k^\times,s\in k$. Now, owing to the assumption on $k$ we note that $B$ is countable. So, let $B=\{g_i:i\in \mathbb{N}\}$ be an enumeration of the elements of $B$ with $g_1=e$. 

\begin{lemma}\label{techlemma}
Let $\beta_1,\ldots,\beta_m\in \Phi^+$ and $s_1,\ldots,s_m\in k$. For each $\alpha_j\in \Phi^+$, let $I^m_j=\{j_1,\ldots,j_{r_j},j_{r_j+1},\ldots,j_{r_j+l_j}\}\subset I_m=\{1,\ldots,m\}$ such that 
\begin{enumerate}
\item $\beta_{j_1}=\beta_{j_2}=\cdots=\beta_{j_{r_j}}=\alpha_j$ and  
\item  $\mathrm{ht}(\beta_{j_{r_j+1}}),\ldots,\mathrm{ht}(\beta_{j_{r_j+l_j}})<\mathrm{ht}(\alpha_j)< \mathrm{ht}(\beta_i)$ for all $i\in I_m\setminus I^m_j$. 
\end{enumerate}
Then  $x_{\beta_1}(s_1)\cdots x_{\beta_m}(s_m) =x_{\alpha_1}(t_1)\cdots x_{\alpha_N}(t_N)$, where
$$ t_j=(s_{j_1}+\cdots+s_{j_{r_j}})+F_j(s_{j_{r_j+1}},\ldots,s_{j_{r_j+l_j}}),$$ for some polynomial $($over $k)$ $F_j(X_1,\ldots,X_{l_j})$ vanishing at zero $(1\leq j\leq N)$.
\end{lemma}

\begin{proof}
We induct on $m$. If $m=1$, then the result clearly holds. Assuming that the result is true for $m-1$, we obtain   $x_{\beta_1}(s_1)\cdots x_{\beta_{m-1}}(s_{m-1})= x_{\alpha_1}(c_1)\cdots x_{\alpha_N}(c_N)$, with $c_j=(s_{j_1}+\cdots+s_{j_{r_j}})+P_j(s_{j_{r_j+1}},\ldots,s_{j_{r_j+l_j}})$ for  some polynomial $P_j(X_1,\ldots,X_{l_j})$ with zero constant term and $I^{m-1}_j=\{j_1,\ldots,j_{r_j+l_j}\}\subset I_{m-1}$ satifies conditions $(1)$ and $(2)$ of the lemma ($1\leq j\leq N$). Assume that $\beta_m=\alpha_n$. By applying Equation \eqref{commutator} (Chevalley's commutator formula), let 
\begin{align}
x_{\alpha_i}(c_i)x_{\alpha_n}(s_m)=x_{\alpha_n}(s_m)x_{\alpha_i}(c_i)D_{i}\;, \;(n+1\leq i\leq N),
\end{align}
where $D_i$ is either equal to $1$ or a product of terms of the form $x_{\alpha}(\lambda c_i^ps_m^q)$ for some $\lambda\in k$ and $p,q$ some positive integers and hence, for any such $\alpha$, we note that $\mathrm{ht}(\alpha)>\mathrm{ht}(\alpha_i), \mathrm{ht}(\alpha_n)$ ($n+1\leq i\leq N$). 
Hence,
 \begin{align} x_{\beta_1}(s_1)\cdots x_{\beta_m}(s_m)&= x_{\alpha_1}(c_1)\cdots x_{\alpha_N}(c_N)x_{\alpha_n}(s_m)\\
&=\left(\prod\limits_{i=1}^{n-1}x_{\alpha_i}(c_i)\right)
x_{\alpha_n}(c_n+s_m)\left(\prod\limits_{i=1}^{N-1-n}
(x_{\alpha_{n+i}}(c_{n+i}))D_i\right) x_{\alpha_N}(c_N)
\end{align}
Note that  none of the $D_i$'s contain a factor of the form $x_{\alpha_{n+1}}(s)$. Now apart from $x_{\alpha_{n+2}}(c_{n+2})$, only $D_1$ possibly contains a factor of the form $x_{\alpha_{n+2}}(dc_{n+1}^ps_m^q)$ for some positive integers $p,q$. So again by repeated application of Equation \eqref{commutator}, we have 
\begin{align*}
\left(\prod\limits_{i=1}^{N-1-n}(x_{\alpha_{n+i}}(c_{n+i}))D_i\right)x_{\alpha_N}(c_N)\\
=x_{\alpha_{n+1}}(c_{n+1})x_{\alpha_{n+2}}(c_{n+2}+dc_{n+1}^ps_m^q)E_1D_2 \left(\prod\limits_{i=3}^{N-1-n}(x_{\alpha_{n+i}}(c_{n+i}))D_i\right)x_{\alpha_N}(c_N),
\end{align*} 
where $E_1$ is obtained from repeatedly applying Equation \eqref{commutator} in order to move $x_{\alpha_{n+2}}(c_{n+2})$ past the factors appearing in $D_1$ until it appears adjacent to $x_{\alpha_{n+2}}(dc_{n+1}^ps_m^q)$. We repeat the above argument with $\alpha_{n+3}$ and get 
\begin{align*}
E_1D_2 \left(\prod\limits_{i=3}^{N-1-n}(x_{\alpha_{n+i}}(c_{n+i}))D_i\right)x_{\alpha_N}(c_N)\\=E_1x_{\alpha_{n+3}}(c_{n+3}+d_1c_{n+1}^{p_1}s_m^{q_1})E_2D_3\left(\prod\limits_{i=4}^{N-1-n}(x_{\alpha_{n+i}}(c_{n+i}))D_i\right)x_{\alpha_N}(c_N)\\
=x_{\alpha_{n+3}}(c_{n+3}+d_1c_{n+1}^{p_1}s_m^{q_1}+d_2c_{n+2}^{p_2}c_{n+1}^{p_3}s_m^{q_2})E_1^\prime E_2D_3\left(\prod\limits_{i=4}^{N-1-n}(x_{\alpha_{n+i}}(c_{n+i}))D_i\right)x_{\alpha_N}(c_N)
\end{align*} where $d_1,d_2\in \mathbb{Z}$, $p_i,q_i\in \mathbb{N}$, $E_2$ is obtained due to moving $x_{\alpha_{n+3}}(c_{n+3})$ past the terms in $D_2$ and $E_1^\prime$ is obtained subsequently from $E_1$. 

The above process is repeated with respect to the subsequent roots $\alpha_{n+4},\ldots,\alpha_N$ inductively to finally obtain
\begin{align}
x_{\alpha_1}(c_1)\cdots x_{\alpha_N}(c_N)x_{\alpha_n}(s_m)=x_{\alpha_1}(t_1)\cdots x_{\alpha_N}(t_N),
\end{align} where $t_i=c_i$ for $i=1,\ldots, n-1$, $t_n=c_n+s_m$ and $t_{n+i}=c_{n+i}+G_i(s_m, c_{n+1},\ldots,c_{i_1})$, where $G_i$ is a polynomial with zero constant term and $\mathrm{ht}(\alpha_{i_1})<\mathrm{ht}(\alpha_{n+i})<\mathrm{ht}(\alpha_{i_1+1})$, $i=1,2,\ldots, N-n$. Since $c_1,\ldots,c_N$ were obtained by invoking the induction hypothesis, it is clear that  $t_1,\ldots,t_N$ above, satisfy the conditions of the lemma. This completes the proof.
\end{proof}

\begin{remark} Note that for Lemma \ref{techlemma}, it is not necessary to impose any characteristic restriction on $k$.
\end{remark}

\begin{lemma}\label{l1}
For every $n\in \mathbb{N}$, there exists a pair $(k_n,\varphi_n)$, where $k_n$ is an algebraically closed subfield of $k$ and $\varphi_n$ is an automorphism of $k_n$ such that the following conditions are satisfied:
\begin{enumerate}
\item $k_n\subset k_{n+1}$ and $\varphi_{n+1}|_{k_n}=\varphi_n$.

\item There exists $y_n\in B(k_n)$ such that $y_n^{-1}\widetilde{\varphi_n}(y_n)=g_n$, where $\widetilde{\varphi_n}$ is an automorphism of $B(k_n)$ induced from $\varphi_n$. 
\end{enumerate}
\end{lemma}
\begin{proof}
We induct on $n\in \mathbb{N}$. For $n=1$, set $k_1=\overline{\mathbb{Q}}$, $\varphi_1=\mathrm{Id}|_{\overline{\mathbb{Q}}}$ and $y_1=e$. We construct $(k_2,\varphi_2)$ as follows:

So let $g_2=\prod\limits_{i=1}^lh_{\alpha_i}(a_i)\prod\limits_{i=1}^N x_{\alpha_i}(b_i)\in B$ ($a_i\in k^\times,b_i\in k$). If $a_1$ is algebraic (respectively, transcendental) over $k_1$, then let $K_1:=k_1$ (respectively, $K_1:=\overline{k_1(a_1)}$). Then there exists an automorphism of $K_1$ which extends $\varphi_1$. Repeating this process with the subsequent scalars $a_2,\ldots,a_l,b_1,\ldots,b_N$, we will finally have an algebraically closed field $k_1^\prime$ such that $k_1(a_1,\ldots,a_l,b_1,\ldots,b_N)\subset k_1^\prime\subset k$, and an automorphism  (say) $\psi$ of $k_1^\prime$, such that $\psi|_{k_1}=\varphi_1$. Note that $\mathrm{tr.deg}_{k_1} k_1^\prime$ is at most $l+N$ and hence $\mathrm{tr.deg}_{k_1^\prime}k$ is countable. The latter observation implies that there  exist $l+N$ elements $t_1,\ldots,t_l,s_1,\ldots,s_N\in k$ which are algebraically independent over $k_1^\prime$. Set $E_1=k_1^\prime(t_1,\ldots,t_l,s_1,\ldots,s_N)$ and we intend to show that a candidate for $k_2$ is an algebraic closure $\overline{E_1}$ of $E_1$ in $k$.
Consider the element $x=\prod\limits_{i=1}^lh_{\alpha_i}(t_i)\prod\limits_{i=1}^N x_{\alpha_i}(s_i)\in B$.
 Then by Equation \eqref{1.5} and Lemma \ref{techlemma}, we  have $xg_2=\prod\limits_{i=1}^lh_{\alpha_i}(a_it_i)\prod\limits_{i=1}^N x_{\alpha_i}((\prod_{j=1}^la_j^{-\langle\alpha_i, \alpha_j\rangle}s_i)+b_i+F_i(s_1,\ldots,s_{n_i}))$, where $F_i(X_1,\ldots,X_{n_i})$ is a polynomial with coefficients in $k_1^\prime$ and $F_i$ vanishes at zero and $n_i$ is such that $\mathrm{ht}(\alpha_{n_i})<\mathrm{ht}(\alpha_i)$ ($1\leq i\leq N$).

Since $a_1,\ldots,a_l\in k^\times$, it follows from Lemma \ref{poly} that there exists an automorphism (say) $\psi^\prime$ of $E_1$ such that $\psi^\prime|_{k_1}=\varphi_1$ and $t_1,\ldots,t_l,s_1,\ldots,s_N$ are transformed under $\psi^\prime$ via the following assignments :

\begin{align*}
t_i&\mapsto a_it_i,~ 1\leq i\leq l\\
s_i&\mapsto \prod_{j=1}^la_j^{-\langle\alpha_i, \alpha_j\rangle}s_i+b_i+F_i(s_1,\ldots,s_{n_i}).
\end{align*}

Let $\varphi_2$ be an extension of $\psi^\prime$ to an algebraic closure $\overline{E_1}$ of $E_1$ in $k$.
Then note that $g_2,x\in B(k_2)$ and the automorphism $\widetilde{\varphi_2}$ of $B(k_2)$ is such that $x^{-1}\widetilde{\varphi_2}(x)=g_2$ as desired. 

Now assume that the pairs $(k_1,\varphi_1),\dots,(k_n,\varphi_n)$ have been constructed for some $n\geq 2$. Then we can construct $(k_{n+1},\varphi_{n+1})$ in exactly the same way as $(k_2,\varphi_2)$ was constructed above. This completes the proof.

\end{proof}

After finding a sequence of pairs $\{(k_n,\varphi_n)\}_{n\in \mathbb{N}}$ as in Lemma \ref{l1}, we note that $\bigcup\limits_{n\in \mathbb{N}}k_n=k$. Indeed, for if $a\in k$ is any scalar, then consider the  element $x_{\alpha}(a)\in B$, for some $\alpha\in \Phi^+$. If $g_m=x_\alpha(a)$ for some $m\in \mathbb{N}$, then by construction, $a\in k_m$.

Now for every $a\in k$ fix a positive integer $n_a$ such that $a\in k_{n_a}$. Then the map $\varphi:k\rightarrow k$ defined by $\varphi(a)=\varphi_{n_a}(a)$, for all $a\in k$, defines an automorphism of $k$. Observe that for any $g_n\in B$, $y_n^{-1}\widetilde{\varphi}(y_n)=y_n^{-1}\widetilde{\varphi_n}(y_n)=g_n$ ($y_n$ as in Lemma \ref{l1}). Hence $R(\widetilde{\varphi})=1$. Thus we have proven the following

\begin{theorem}\label{abstbor}

Let $k$ be an algebraically closed field of countable transcendence degree over $\mathbb{Q}$ and $B$ a Borel subgroup of a simple algebraic group over $k$. Then there exists an abstract automorphism of $B$ such that the associated twisted conjugacy action of $B$ on itself is transitive.
\end{theorem}
\begin{remark}
 If $k$ is an algebraically closed field of infinite transcendence degree over $\mathbb{Q}$, then $\mathrm{GL}_n(k)$ admits an abstract automorphism $\varphi$ for which $R(\varphi)=1$ \cite[Theorem 7]{timurjaa}. As a first step, the theorem is proven under the assumption that $\mathrm{tr.deg}_\mathbb{Q}k$ is countable \cite[Theorem 6]{timurjaa}. The general case is subsequently argued by invoking the L\"{o}wenheim-Skolem Theorem (c.f. \cite[Theorem 5]{timurjaa}) from model theory. Using similar arguments, it is perhaps possible to show that Theorem \ref{abstbor} (above) holds even if $\mathrm{tr.deg}_\mathbb{Q}k$ is uncountable but we refrain from recording it here since the relevant model theoretic set-up is not entirely clear to us.
 
 % He used a similar technique to prove the same results for the groups $\mathrm{SO}_n(k), \mathrm{Sp}_{2n}(k)$ (see \cite{nas20}). 
 %One can prove analogues result (using model theory) for Borel subgroups of a simple algebraic groups defined over an algebraically closed field $k$ of uncountable transcendence degree over $\mathbb{Q}$.
\end{remark}

\section*{Acknowledgements}

 We thank the reviewers for their careful reading of the manuscript and their valuable comments, which helped in improving the exposition considerably. Thanks are due to a reviewer  of an earlier draft of the manuscript, for observing Theorem \ref{general}, and the alternative argument provided at the end of Example 4(b) in Section \ref{ex}. We also thank Wilberd van der Kallen, Vladimir L. Popov, Anupam Kumar Singh and Maneesh Thakur for their encouraging remarks on this work.

\end{document}